\documentclass[preprint]{imsart}

\usepackage{amsthm,amsmath,amsfonts,amssymb}
\RequirePackage[colorlinks,citecolor=blue,urlcolor=blue]{hyperref}

\startlocaldefs
\usepackage[dvipsnames]{xcolor}
\usepackage[utf8]{inputenc}
\usepackage{soul}
\newcommand{\norm}[2]{\left\|#1\right\|_{#2}}
\newcommand{\scalar}[2]{\langle{ #1},{#2} \rangle}

\newcommand{\set}[1]{\left\{#1\right\}}
\newcommand{\abs}[1]{\left\lvert #1 \right\rvert}
\newcommand{\lr}[1]{\left(#1\right)}
\newcommand{\nat}{\mathbb N}

\newcommand{\af}{H_\varphi}
\newcommand{\lpsi}{\Lambda^g_\psi}
\newcommand{\lmg}{\Lambda^{g}}
\newcommand{\lmf}{\Lambda^{f}}
\newcommand{\Ck}{C_{k}}
\newcommand{\yn}{Y^{n}}
\newcommand{\opA}{A}
\newcommand{\opT}{T}
\newcommand{\opG}{G}
\newcommand{\opK}{K}
\newcommand{\spZ}{Z}
\newcommand{\asta}{A^\ast A}
\newcommand{\eS}[1]{\mathcal S^{#1}}
\newcommand{\spc}{\operatorname{SPC}}
\newcommand{\E}{\mathbb E}

\newcommand{\veps}{\varepsilon}
\newcommand{\tr}{\operatorname{tr}}
\newcommand{\omf}{\omega_{f_0}}
\newcommand{\myspan}{\operatorname{span}}
\newcommand{\myrank}{\operatorname{rank}}

\newcommand{\ra}{r_{\alpha}}

\newcommand{\tast}{t_{\delta}}
\newcommand{\nast}{{k_{\delta}}}
\newcommand{\kn}{k_{n}}
\newcommand{\Sn}[1]{S\lr{#1,\rho}}
\newcommand{\bigo}{\mathcal O}

\newcommand{\X}{\mathcal{X}}
\newcommand{\UU}{U}
\newcommand{\Sk}{\mathcal{X}_{k}}
\newcommand{\Skd}{\mathcal{X}_{\nast}}
\newcommand{\Pg}{\Pi_{k}^{Y}}
\newcommand{\Pf}{\Pi_{k}^{X}}

\newcommand{\ktn}{k(n)}
\newcommand{\Pftn}{\Pi_{\ktn}^{X}}
\newcommand{\Xk}{X_{k}}
\newcommand{\Xkn}{X_{\kn}}

\theoremstyle{plain}
\newtheorem{thm}{Theorem}
\newtheorem{lem}{Lemma}[section]
\newtheorem{prop}{Proposition}[section]
\newtheorem{cor}{Corollary}[section]

\theoremstyle{definition}
\newtheorem{ass}{Assumption}
\newtheorem{de}{Definition}
\theoremstyle{remark}
\newtheorem{rem}{Remark}
\newtheorem{xmpl}{Example}
\newtheorem*{xmplno}{Example}

\newcommand{\ca}{C_1} 
\newcommand{\cb}{c} 
\newcommand{\cc}{C_2} 
\newcommand{\cd}{C_3} 
\newcommand{\ce}{C_6} 
\newcommand{\cff}{C_7} 
\newcommand{\cg}{C_4} 
\newcommand{\ch}{C_5} 
\newcommand{\ci}{C_8} 

\endlocaldefs

\begin{document}

\begin{frontmatter}

\title{Designing truncated priors for direct and inverse Bayesian problems}

\runtitle{Truncated priors for Bayesian problems}

\begin{aug}
\corref{Sergios Agapiou}

\author{\fnms{Sergios}
  \snm{Agapiou}\corref{}\thanksref{a}\ead[label=e1,mark]{agapiou.sergios@ucy.ac.cy}}
\and
\author{\fnms{Peter} \snm{Math{\'e}}\thanksref{b}\ead[label=e2]{peter.mathe@wias-berlin.de}}
\address[a]{Department of Mathematics and Statistics, University of Cyprus, \href{mailto:agapiou.sergios@ucy.ac.cy}{agapiou.sergios@ucy.ac.cy}}
\address[b]{Weierstrass Institute for Applied Analysis and Stochastics, Berlin, \href{mailto:peter.mathe@wias-berlin.de}{peter.mathe@wias-berlin.de}}

\runauthor{S. Agapiou and P. Math\'e}
\affiliation{University of Cyprus and Weierstrass Institute}

\end{aug}
 \begin{abstract}
The Bayesian approach to inverse problems with functional unknowns,
has received significant attention in recent years. An important component of
the developing theory is the study of the asymptotic performance of
the posterior distribution in the frequentist setting. 
The present paper contributes to the area of Bayesian inverse problems by formulating a posterior contraction theory for linear inverse problems, with truncated Gaussian series priors, and under general smoothness assumptions.
Emphasis is on the intrinsic role of the truncation point both for the direct as well as for the inverse problem, which are related through the modulus of continuity as this was recently highlighted by Knapik and Salomond (2018). 
\end{abstract}




\end{frontmatter}

\section{Outline}
\label{sec:outline}
We study the problem of recovering an \emph{unknown function} $f$ from a
noisy and indirect \emph{observation} $Y^{n}$. In particular, we consider a
class of inverse problems in Hilbert space, given as
\begin{equation}
  \label{eq:Ys}
  Y^{n} = \opA f + \frac 1 {\sqrt n} \xi.
\end{equation}
Here~$\opA\colon X \to Y$ is a linear mapping between two separable Hilbert spaces~$X$ and~$Y$, termed \emph{the forward operator}. 
For our analysis, we shall assume that the mapping~$\opA$ is compact and injective. It will be clear from the assumptions made later, that the injectivity can easily be relaxed. These assumptions will also entail the compactness of~$\opA$.
The \emph{observational noise} is assumed to be additive, modeled as a Gaussian white noise~$\xi$ in the space $Y$, scaled by $\frac{1}{\sqrt{n}}$. The problem of recovering the unknown $f$ from the observation $Y^{n}$ is assumed to be
ill-posed, in the sense that $\opA$ is not continuously invertible on its range~$\mathcal R(\opA)\subset Y$. 
In particular, this means that~$\mathcal R(\opA)$ is not contained in a finite-dimensional subspace. Notice that although the white noise $\xi$ can be defined by
its actions in the space $Y$, it almost surely does not belong to~$Y$. Rigorous meaning to model~\eqref{eq:Ys} can be
given using the theory of stochastic processes, see Section~6.1.1 in~\cite{MR3588285}.  

In the Bayesian approach to such inverse problems, we postulate a
prior distribution $\Pi$ on $f$ and combine with the (Gaussian) data
likelihood $P^{n}_f$ to obtain the posterior distribution
$\Pi(\cdot|Y^{n})$ on $f|Y^{n}$, see~\cite{MR3839555} for a comprehensive overview of the area. We are interested in studying the
frequentist performance of the posterior distribution in the small
noise asymptotic regime~$\frac 1 {\sqrt n}\to 0$, and hence~$n\to\infty$. More specifically, we consider
observations generated from a fixed underlying element
$f_{0}\in X$, $Y^{n}\sim P^{n}_{f_0}$, and study rates of contraction
of the resulting posterior distribution around $f_{0}$, as $n\to\infty$.

The study of rates of posterior contraction for inverse problems has received great attention in the last decade, initiated by~\cite{MR2906881}. The authors of that study considered Gaussian priors
which were conjugate to the Gaussian likelihood. This results in Gaussian posteriors, having explicitly known posterior mean and covariance operator. Moreover, by assuming that the prior covariance
operator and the linear map $\opA$ are mutually diagonalizable, the
infinite dimensional inverse problem was reduced to an infinite
product of one-dimensional problems. In this way, 
posterior
contraction rates could be determined using explicit calculations both
for moderately, and in the subsequent studies~\cite{MR3031282} and
\cite{ASZ14}, for severely ill-posed linear forward operators.  This approach
was surveyed and extended to general ill-posedness of the linear operator by the present authors in~\cite{MR3815105},
using techniques from regularization theory.

Several works extended the diagonal linear Gaussian-conjugate setting
to various other directions, for example~\cite{ALS13} and~\cite{MR3985479}
studied linear forward operators which are not simultaneously
diagonalizable with the covariance operator of the Gaussian prior, and
\cite{MR3535664} studied linear hypoelliptic pseudo-differential
forward operators with Gaussian priors.

More recently, there has been a wealth of contributions in more complex inverse problems, including non-linear ones arising in PDE models, see for example~\cite{MR4118619, MR4151406, kweku19}.
 Another line of progress has been the consideration of more general priors, so far for linear inverse problems, see~\cite{KR13} and~\cite{MR4116718}. 
 The idea underlying all of these works, is to first establish rates of contraction for the related direct problem
\begin{equation}
  \label{eq:direct}
  Y^{n} = g + \frac 1 {\sqrt n} \xi,
\end{equation}
with unknown  $g=\opA f$, in which the data~$\yn$ are generated from~$g_0 = \opA f_0$.
Once such rates are established, the strategy is
to control distances on the level of $f$ by distances on the level of $g$, when restricting on a sieve set $S_n$ on which the inversion of $\opA$ is well-behaved. This enables to translate rates for the direct problem to rates for the inverse problem when the posterior is restricted on the sieve set $S_n$. If the posterior mass concentrates on $S_n$, then these rates are also valid for the unrestricted posterior. In order to establish direct rates, the authors of the above-mentioned studies use the testing approach, see~\cite{MR2332274}.

Here we shall explore the methodology proposed by~\cite{MR3757524}, which explicitly uses the \emph{modulus of continuity} (function) in order to translate rates for the direct to rates for the inverse problem. This approach is in principle general, however, so
far it has been applied to 
certain linear inverse problems, with moderately and severely ill-posed forward operators, under Sobolev-type smoothness assumptions on the truth $f_{0}$. Our work is also related to~\cite{MR4116718}, in that both works use approximation-theoretic techniques to control the inversion or $\opA$.

We consider (centered) Gaussian priors on $f$, arising by truncating the series representation of an underlying infinite-dimensional prior on the separable Hilbert space $X$, see e.g.~\cite[Sect.~2.4]{MR3839555}. 
We develop a comprehensive theory for establishing rates of contraction for general linear inverse problems, under general smoothness conditions, with a particular focus on the optimal choice of the truncation level. Truncated priors are both practically relevant since when implementing one needs to truncate, but also can lead to optimal rates of contraction for a smoothness-dependent choice of the truncation level as a function of $n$, see e.g.~\cite{MR2418663}. Furthermore, in \cite{KR13}, it was shown that putting a hyper-prior on the truncation level can lead to adaptation to unknown smoothness. This was done in the context of inverse problems with specific types of smoothness (Sobolev) and degree of ill-posedness of the operator (power or exponential type). See also~\cite{MR3091697}, where direct models are studied. The extension of adaptation to the general framework which we consider here, is interesting but beyond the scope of this work.

 Contraction rates for the problems~\eqref{eq:Ys} and~\eqref{eq:direct} are related through the modulus of continuity of the mapping~$\opA^{-1}$. Thus, knowing a contraction rate, say~$\delta_n$ for the direct problem~\eqref{eq:direct}, and knowing the behavior of the modulus of continuity~$\omega_{f_0}(\opA^{-1},S_n,\delta),\ \delta>0$, where~$S_n$ is the (finite-dimensional) support of the prior, we obtain a contraction rate for~$f_0$ as~$\omega_{f_0}(\opA^{-1},S_n,\delta_n),\ n\to\infty$. In this program, the role of the truncation level~$k=k(n)$ is most important. There is $k^{(1)} =k^{(1)}(n)$ that should be used for the inverse problem, $k^{(2)}=k^{(2)}(n)$ which works for the direct problem, and finally~$k^{(3)}=k^{(3)}(n)$ used in the modulus of continuity. For the plan, as outlined above, to work we need to establish that actually a universal choice~$k=k(n)$ is suited for all three problems.

In Section~\ref{sec:setting} we shall introduce the overall setting of the study, and we shall formulate Theorem~\ref{thm:direct-inverse}, which comprises the main achievements of this study. The rest of the study is composed of four parts.

In Section~\ref{sec:direct} we will develop the tools needed to analyze the direct problem~\eqref{eq:direct} and obtain~$k^{(2)}(n)$ depending on the underlying prior covariance and the smoothness of $g_{0}$. Due to linearity, Gaussian priors
on~$f$ induce Gaussian priors on~$g=\opA f$, which streamlines the
analysis of 
problem~\eqref{eq:direct}. However, the smoothness of the induced true element in the
direct problem, $g_0=\opA f_0$, depends on the smoothing
properties of the operator $\opA$ and in particular, $g_0$ might not
belong to any of the  standard smoothness classes. For this reason, we shall study rates of contraction in the (direct) white noise model, given a Gaussian prior on $g$ and under general smoothness assumptions on~$g_{0}$. 
Emphasis will be given on the construction of the prior. We shall analyze truncated Gaussian priors posed directly on~$g$, obtained by truncating the series representation of an underlying infinite-dimensional Gaussian prior (called 'native', below), but also priors that are obtained as linear transformations of truncated Gaussian priors chosen for some~$f$ (called 'inherited', below). The former is relevant in the context of~\eqref{eq:direct} when $\opA$ commutes with the covariance operator of the underlying Gaussian prior on $f$. In the latter non-commuting case the analysis is more involved and restrictive.
This section is self-contained and may be of independent interest. The main result is Theorem~\ref{thm:spc-bound} and it includes a way of assessing the optimality of the obtained bounds.

In Section~\ref{sec:inverseP} we introduce the modulus of continuity, and we shall discuss its behavior, as~$\delta\to 0$, under an approximation-theoretic perspective. The main result here will be  Theorem~\ref{thm:phi-theta-bound}, indicating the choice~$k^{(3)}(n)$.

In Section~\ref{sec:relating} we shall show that the choice~$k^{(2)}(n)$ yields optimal behavior also of the modulus of continuity, such that we may let~$k^{(2)}=k^{(3)}$. Therefore, letting~$k^{(1)}(n) = k^{(2)}(n) = k^{(3)}(n)$ yields a contraction rate for the inverse problem allowing us to establish the main result, 
Theorem~\ref{thm:direct-inverse}.

We exemplify the obtained (general) bounds at 'standard instances', with forward operators which have a moderate decay of singular numbers, a (severe) exponential decay, but also a (mild) logarithmic decay in Section~\ref{sec:xmpls}. 
Many examples for such instances are known. The Radon transform is prototypical for a power type decay of singular numbers, see the monograph~\cite{MR1847845}. The heat equation is known to exhibit an exponential decay of the singular numbers, see~\cite{MR1408680}, which is also a good resource for more examples.
In particular, we explicitly derive (minimax) contraction
rates under Sobolev-type smoothness, both for the direct and inverse problems, in the above-mentioned instances.

In order to streamline the presentation the proofs of the results are given separately in Section~\ref{sec:proofs}. 

\section{Setting and main result}
\label{sec:setting}
We next define certain concepts that will be needed for the development of the paper. After establishing some notation, 
we introduce rates of posterior contraction for the direct and inverse problem, links between the main operators pertaining to our analysis, as well as the concept of smoothness that  will be used throughout the paper. We formulate the main result in Section~\ref{sec:main-result}.

\subsection{Notation}
We shall agree upon the following notation. We denote by $\norm{\cdot}{X}, \norm{\cdot}{Y}$ the norms in $X,Y$, respectively. When there is no confusion we will use plainly $\norm{\cdot}{}$ and the same notation will be used for the operator norm in $X$ or $Y$. For a (compact
self-adjoint) linear operator, say~$\opG\colon X\to X$ we denote by~$s_{j}(\opG),\
j=1,2,\dots$ the non-increasing sequence of its singular numbers. 
We reserve the notation~$s_{j} = s_{j}(H),\ j=1,2,\dots$ for the operator~$H:=\asta$, the self-adjoint companion to the mapping~$\opA$.

Furthermore, according to whether we study the inverse
problem~(\ref{eq:Ys}) or the related direct problem~(\ref{eq:direct})
we shall denote elements by~$f$ or~$g$ ($f_{0}, g_{0}$ for the corresponding true elements).

For two sequences $(a_n)$ and $(b_n)$ of real numbers, $a_n\asymp b_n$ means that $|a_n/b_n|$ is bounded away from zero and infinity, while $a_n\,\lesssim\,b_n$ means that $a_n/b_n$ is bounded from above.

\subsection{Prior distribution and posterior contraction}
\label{sec:contraction}
We shall use priors $\Pi$ which are truncations of a Gaussian prior~$\mathcal N(0,\Lambda)$, for a self-adjoint, trace-class and positive definite covariance operator $\Lambda$. Such priors are characterized by the underlying covariance~$\Lambda$ and the truncation level~$k$. Below we shall use the notation $\Lambda=\lmf$ for Gaussian priors on $f$ in the context of \eqref{eq:Ys} and $\Lambda=\lmg$ for Gaussian priors on $g$ in the context of \eqref{eq:direct}. 

\begin{xmplno}[$\alpha$-regular prior]
Given~$\alpha>0$, we  call the prior~$\mathcal N(0, \Lambda)$ $\alpha$\emph{-regular}, if the singular numbers~$s_{j}(\Lambda)$ 
decay like~$s_{j}(\Lambda)\asymp j^{- (1 + 2\alpha)},\ j=1,2,\dots$
\end{xmplno}

Let us fix a prior distribution~$\Pi$ on the unknown~$f$, and consider
data~$Y^{n}$ generated from the model~(\ref{eq:Ys}) for a fixed true element $f_0\in X$, $Y^n\sim P^n_{f_0}$. We are interested in deriving rates of contraction of the posterior $\Pi(\cdot|Y^{n})$ around~$f_{0}$, in the small noise limit $n\to\infty$.
In particular, we find sequences $\veps_n\to0$ such that, for an arbitrary sequence $M_n\to\infty$, it holds 
\begin{equation}
\E_{0}\Pi\lr{f, \ \norm{f - f_{0}}{X} > M_{n}\veps_{n}|
  Y^{n}}\to0\label{it:inverse}.
\end{equation}
Here $\E_{0}$ denotes expectation with respect to $P^n_{f_0}$. 

One can also derive rates of posterior contraction around $\opA f_{0}$, that is sequences $\delta_n\to0$, such that for arbitrary $M_n\to\infty$
\begin{equation}
  \E_{0}\Pi\lr{f, \ \norm{\opA(f - f_{0})}{Y} > M_{n}\delta_{n}|
  Y^{n}}\to0\label{it:direct}.
\end{equation}

Such rates $\delta_n$ and $\veps_n$ will be called rates of contraction for the \emph{direct} and \emph{inverse} problem, respectively. We are going to derive rates of contraction for the inverse problem, by deriving rates of contraction  for the direct problem and using the modulus of continuity as was proposed in~\cite{MR3757524}. These rates of contraction will be obtained by means of the \emph{squared posterior contraction}, a concept which will be introduced in detail in~\S~\ref{sec:direct}.


\subsection{Relating operators in Hilbert space}
\label{sec:relating-ops}
As highlighted in the introduction we deal with several
operators. In order to facilitate our analysis we need to relate these operators and to this end we introduce the following concept.

\begin{de}
  [index function] \label{de:index-noncomm}
  A function~$\rho\colon (0,\infty)\to (0,\infty)$ is called \emph{index function} if it
  is continuous, non-decreasing, with~$\rho(0+)=0$. For an index function~$\rho$, we
assign the companion~$\Theta_{\rho}(t) := \sqrt t \rho(t),\ t>0$.
\end{de}

The primary operator we deal with is the forward operator~$\opA:X\to Y$ which governs
equation~(\ref{eq:Ys}). Its self-adjoint companion~$H= \asta:X\to X$
will have the central role in our analysis. We mention the following identity:
\begin{equation}\label{eq:a-asta12}
\norm{\opA f}{Y} =\norm{\lr{\asta}^{1/2}f}{X} = \norm{H^{1/2}f}{X}, 
f\in X.  
\end{equation}

Furthermore, in order to obtain rates of contraction for inverse problems from rates of contraction for direct problems, we will need to link the underlying (untruncated)
covariance operator~$\lmf$ of the Gaussian prior
for~$f$ to the operator~$H$.  We will study two cases. Initially we shall
assume that~$\lmf$ and~$H$ commute. Precisely, we impose the following
assumption: 

\begin{ass}[prior in scale]
\label{ass:link-noncomm}
  There is an index 
  function~$\chi$ such that
  $$
  \lmf= \chi^{2}(H).
  $$  
\end{ass}

This commutative case allows for a general analysis, but has limited
applicability, as it may be hard to design a truncated  Gaussian
prior, because the singular basis of~$H$ (and hence $\lmf$) may not be known.

Instead, we may relax the commutativity assumption, and impose a
corresponding link condition.
\begin{ass}
  [prior linked to scale]\label{ass:prior-linked2scale-noncomm}
  There is some exponent~$a\geq 1/2$ such that
  $$
  \norm{\lr{\lmf}^{1/2}f}{X} \asymp \norm{H^{a}f}{X},\quad f\in X.
  $$
\end{ass}
The requirement~$a\geq 1/2$ has a natural reason. We need to link the operators~$\opA$ and~$\lmf$ in several places, and by virtue of~\eqref{eq:a-asta12} this can be done via $H^{1/2}$. Therefore, the case $a=1/2$ needs to be covered in the assumption.

We mention, that within the non-commuting case we confine to power
type links. 
Also, notice that~$\chi(t) = t^a$ in Assumption~\ref{ass:link-noncomm} yields a special instance of Assumption~\ref{ass:prior-linked2scale-noncomm}.
One may extend to more general index functions, but
for the sake of simplicity we do not pursue this direction here.

Both Assumptions~\ref{ass:link-noncomm} and~\ref{ass:prior-linked2scale-noncomm} have impact on the mapping properties of~$\opA$. First, the mappings~$H$ and~$\lmf$ share the same null spaces (kernels). Also, since the covariance operator~$\lmf$, being trace-class,  is compact,  this compactness transfers to~$H$, and a fortiori to~$\opA$. Thus, under these links the compactness of~$\opA$ cannot be avoided, while its injectivity can be relaxed by factoring out the common null spaces.

\begin{rem}
In our analysis the self-adjoint companion~$H$ to the operator~$\opA$ plays the role of the central operator. When studying contraction rates for the inverse problem \eqref{eq:Ys}, smoothness will be given with respect to it. Instead, one might give this role to the operator~$\lmf$, and consider smoothness with respect to this operator. The analysis would be similar, and some results in this direction are given in~\cite[Sec. 5]{MR4116718}.
\end{rem}

 
\subsection{Smoothness concept}
\label{sec:smoothness}

For the subsequent analysis it will be convenient to introduce the
smoothness of an element~$h\in \spZ$ in a Hilbert space~$\spZ$,  with respect to some injective
positive self-adjoint operator, say~$\opG \colon \spZ \to \spZ$, in terms of
\emph{general source conditions}.

\begin{de}
  [Source set]\label{def:gen-source}
  Given a positive-definite, self-adjoint operator~$\opG$, and an index
  function~$\rho$ the set
  \begin{equation}
    \label{eq:gen-source}
 \opG_{\rho}:= \set{h, \quad h= \rho(\opG)v,\quad \norm{v}{\spZ}\leq 1}.    
\end{equation}
is called a source set.
\end{de}
\begin{rem}
  The sets~$\opG_{\rho}$ from above are ellipsoids in the Hilbert
  space~$\spZ$. The element~$v$ is often called \emph{source element},
  and the representation~$h= \rho(\opG)v$ is called~\emph{source-wise
    representation}.   
  We emphasize that elements in~$\opG_\rho$ are in the range of~$\rho(\opG)$, such that for the subsequent analysis Douglas' Range Inclusion Theorem, see its formulation in~\cite{MR3985479}, will be used several times.
  It is seen from~\cite{MR2384768} that, given the injective
  operator~$\opG$ each element~$h\in \spZ$ has a source-wise
  representation for some 
  index function~$\rho$.
\end{rem}

Below, we shall use this concept for specific operators, and specific
functions. For instance, the set
\begin{equation}
  \label{eq:lpsi}
\lpsi:= \lr{\lmg}_\psi\subset Y
\end{equation}
will correspond to a source set for the operator~$\opG:= \lmg\colon Y
\to Y$, and the index function~$\psi$. In some cases we will assume that the index function~$\psi$ is \emph{operator
  concave}. The formal definition is given in Section \ref{sec:proofs}, but we refer to~\cite[Chapt.~X]{MR1477662} for a comprehensive
discussion. Here we mention that a power type
index function~$\psi(t) = t^{a}$ with~$a>0$, is concave exactly if it
is operator concave, hence~$0 < a \leq 1$.

\begin{xmplno}[Sobolev-type smoothness]
  \label{xmpl:es-beta}
  Let~$u_{1},u_{2},\dots$ be the eigenbasis of the compact self-adjoint operator~$\opG$,
  arranged such that the corresponding eigenvalues are non-increasing; this example can be considered either in $X$ or $Y$.
  Given some~$\beta>0$ we consider the \emph{Sobolev-type} ellipsoid
  \begin{equation}
    \label{eq:sobolev-ell}
    \eS \beta(R) := \set{h,\quad  \sum_{j=1}^{\infty}
j^{2\beta}\abs{\scalar{h}{u_{j}}}^{2}\leq R^{2}}.
\end{equation}
Now, suppose that the singular values of~$\opG$ decay as~$s_j(\opG) \asymp j^{-\gamma}$
for some~$\gamma>0$. Then it is a routine matter to check that~$h\in \eS
\beta$ yields that~$h\in \opG_{\rho}$ for an index function~$\rho(t) \propto
t^{\beta/\gamma},\ t\geq 0$,
see~\cite[Prop.~2]{MR3985479}. Similarly the converse holds true, and there is thus a one-to-one correspondence between
Sobolev-type ellipsoids and power-type source-wise representations for
such operators~$\opG$.
\end{xmplno}

\subsection{Main result}
\label{sec:main-result}
We aim at deriving posterior contraction rates for the inverse problem \eqref{eq:Ys}, from contraction rates for the corresponding 
direct problem \eqref{eq:direct} by using the modulus of continuity, for truncated Gaussian priors.
The, truncated at level~$\kn$, Gaussian prior on $f$, has all its mass on a finite dimensional subspace~$\Xkn$, and so does the posterior through the linear model~\eqref{eq:Ys}. 

The following result links the rates of posterior contraction
corresponding to the inverse problem~\eqref{eq:Ys} and the direct problem~\eqref{eq:direct}. It is an
immediate corollary of~\cite[Theorem 2.1]{MR3757524}. 
  \begin{prop}\label{prop:ks}
   Assume we put a, truncated at level $k_n$, Gaussian prior on $f$.
   Let $\delta_n\to0$ be a rate of contraction for the direct
   problem~\eqref{eq:direct} around $g_0=~\opA f_0\in~Y$, for some $f_0\in X$. Then $\veps_n:=\omega_{f_0}(H^{-1/2},\Xkn,\delta_n)$, where $H=\asta$, 
is a rate of contraction for the inverse problem~\eqref{eq:Ys}, at~$f_0$.
\end{prop}
We can thus obtain contraction rates $\veps_n$ for
  the inverse problem by obtaining rates $\delta_n$ for the direct
  problem~\eqref{eq:direct}, and bounds for the inherent modulus of continuity for the inverse problem. The main result of the study implements this program in a general setting with a specific choice of the truncation level~$\kn$.
\begin{thm}\label{thm:direct-inverse}
Consider the inverse problem \eqref{eq:Ys}, recall $H=\asta$,
and
suppose that~$f_{0}$ has smoothness~$H_\varphi$. Assume we put a
truncated Gaussian prior $\mathcal{N}(0, P_{\kn}\lmf P_{\kn})$ on $f$,
with $\lmf$ a self-adjoint, positive-definite, trace-class, linear
operator in $X$, and $P_{\kn}$ the singular projection of $\lmf$. We specify the related (covariance) operator~$\lmg=A\lmf A^\ast$.

Under
\begin{enumerate}
    \item[-] Assumption \ref{ass:link-noncomm}, or
    \item[-] Assumption~\ref{ass:prior-linked2scale-noncomm} with $\mu\leq a$,
\end{enumerate} 
where for the latter assumption we specify~$\chi(t) = t^a$, and~$\varphi(t) =t^{\mu}$, consider the index function
\begin{equation}
  \label{eq:psi-def}
  \psi(t) =\Theta_{\varphi}(\lr{\Theta_{\chi}^{2}}^{-1}(t)), \quad t>0.
\end{equation}

For the choice~$\kn$ according to
\begin{equation}\label{eq:kn-def}
   \kn := \max\set{j,\quad \psi^{2}\lr{s_{j}(\lmg)} > \max\set{\psi^2\lr{\frac 1 n },\frac j n}},
\end{equation}   
let $\delta_n$ be given as
\begin{equation}\label{eq:spc-main-thm}
   \delta_n:= C \max\set{\psi^{2}\lr{\frac 1 n},\frac {\kn} n} 
\end{equation}
for some constant~$C$. 
Then the posterior contracts around $f_{0}$ at a rate 
  \[
  \veps_n\asymp \varphi(\Theta_\varphi^{-1}(\delta_n)),\quad n\to\infty.
  \]
\end{thm}

The strategy for proving Theorem~\ref{thm:direct-inverse} is loosely outlined at the end of Section~\ref{sec:outline}. A main component of both the result and its proof, is the fact that the truncation levels~$\kn$, as given in~\eqref{eq:kn-def}, optimize both the rates $\delta_n$ for the direct problem as well as the bounds on the modulus of continuity. These considerations can be found in §~\ref{sec:relating}, where we establish the steps for proving Theorem~\ref{thm:direct-inverse}.


\section{Direct signal estimation under truncated Gaussian priors} 
\label{sec:direct}

Here we consider the Bayesian approach to signal estimation under
white noise in the space $Y$, that is, the model
\begin{equation}
  \label{eq:noise-model}
  \yn = g + \frac 1 {\sqrt n} \xi,
\end{equation}
where~$\xi$ is Gaussian white noise in $Y$.

For linear Gaussian models with Gaussian priors, it is convenient to describe posterior contraction in terms of the \emph{squared posterior contraction (SPC)}, which by Chebyshev's inequality, is the square of \emph{a} rate of contraction. 

  For an element~$g_{0}\in Y$,
given data~$\yn$, and a truncation level~$k$ for the
(Gaussian) prior, we assign 
\begin{equation}\label{eq:spc-def}
    \spc(k,g_{0}) := \mathbb E^{g_{0}} \mathbb E^{\yn}_k \norm{g_{0} - g}{Y}^2,
\end{equation}
where the inner expectation is with respect to the (Gaussian) posterior 
distribution, whereas the outer expectation concerns the sampling
distribution, given the element~$g_{0}$, that is,\ the data generating
distribution. The $\spc$ for (regularized) untruncated Gaussian priors in the context of (linear) inverse problems, was
analyzed in the previous study~\cite{MR3815105}. Here we develop a
similar approach for direct problems with truncated Gaussian priors, and we will exhibit
some features, as these are specific for the latter.

Having fixed a class~$\mathcal
F\subset Y$, and given the truncation level~$k$, we assign
\begin{equation}
   \spc^{\mathcal F}(k) := \sup_{g_{0}\in \mathcal F} \spc(k,g_{0}),\label{ali:spc-uniform}  
\end{equation}
which is a squared rate of contraction, holding uniformly over the class~$\mathcal F$.

\subsection{Native and inherited Gaussian priors}
In its simplest form, a centered truncated Gaussian prior for~$g$ can be defined
using some orthonormal system, say~$y_{1},y_{2},\dots$ in~$Y$, independent and identically distributed standard
Gaussian random variables~$\gamma_{1},\gamma_{2},\dots$, and  a
square summable positive sequence~$\sigma_{1},\sigma_{2},\dots$, as
\begin{equation}
  \label{eq:prior-native}
  \Pg :=\mathcal{L}\Big( \sum_{j=1}^{k}\sigma_{j} \gamma_{j}y_{j}\Big).
\end{equation}
The square summability of the sequence~$\sigma_{j},\ j=1,2,\dots$ ensures
that the prior~$\Pg$ is the (singular) projection of an
infinite dimensional prior supported in $Y$, having 
finite-trace covariance operator~$\lmg = \sum_{j=1}^{\infty} \sigma_{j}^{2}
y_{j }\otimes y_{j}$. Hence, the prior~$\Pg$ has
covariance~$\Ck=Q_{k}\lmg Q_{k}$, where $Q_{k}$ are orthogonal projections
onto~$\myspan\set{y_{1},\dots,y_{k}}$. We shall call this a
\emph{native (truncated) prior} for~$g$.

On the other hand, a centered finite dimensional Gaussian prior for~$g$ may be defined using a linear transformation of some native truncated prior~$\Pf$
for~$f\in X$, defined along some orthonormal system, say~$x_{1},x_{2},\dots$,
and with corresponding projections~$P_{k}$ onto $\Xk:=\myspan\set{x_{1},\dots,x_{k}}$, thus having
covariance~$P_{k}\lmf P_{k}$. The prior~$\Pg$ for~$g\in Y$
is then obtained as the push forward~$\opT_{\sharp \lr{\Pf}}$ under some linear
mapping~$\opT \colon X \to~Y$, and is supported on~$\opT \Xk$. The Gaussian prior $\Pg$ will thus have
covariance~$\Ck = \opT P_{k}\lmf P_{k} \opT^{\ast}$, and we shall call this
an \emph{inherited (truncated) prior}. Inherited priors are relevant for example when studying the direct problem \eqref{eq:direct} associated to the inverse problem~\eqref{eq:Ys}. When using such an inherited prior, we will quantify the relation between the mapping~$\opT$ and the covariance operator~$\lmf$ driving~$\Pf$, in order to control the effect of~$\Ck$.




In this context we shall measure the smoothness of the 
truth~$g_0$ relative to the covariance operator~$\lmg$ of the underlying infinite dimensional Gaussian prior on $g$, and we shall
assume the smoothness condition~$g_{0}\in\lpsi$ for some index
function~$\psi$, see~(\ref{eq:lpsi}). For inherited priors, the operator~$\lmg$ will be given
as the covariance of the push forward of the underlying infinite dimensional prior on~$f$,~$\lmg= \opT \lmf \opT^{\ast}$. We stress that for inherited priors in general
we cannot ensure that the covariance $C_k$ corresponds to the singular projection of $\lmg$, that is that~$ \opT P_{k} \lmf P_{k}
\opT^{\ast}=Q_{k}\lmg Q_{k}$ or equivalently that $\opT \Xk$ coincide with the singular spaces of $\lmg$; see the next subsection for details. Nevertheless, we still say that $C_k$ is truncated at level $k$, since it has rank $k$.



\subsection{Basic SPC bound}
We shall start with proving a basic bound on the squared posterior contraction as given in~\eqref{ali:spc-uniform} in the white noise model~\eqref{eq:noise-model}, for both native and inherited truncated Gaussian priors~$\mathcal N(0,C_{k})$, under general smoothness on the truth.

When treating inherited priors, it will be important that the
projections~$P_{k}$ in the corresponding $C_{k}$, are along the singular spaces of~$\lmf$, such
that~$P_{k} \lmf P_{k} = \lr{\lmf}^{1/2} P_{k} \lr{\lmf}^{1/2}$, and
similarly,~$\Ck = \opT \lr{\lmf}^{1/2} P_{k} \lr{\lmf}^{1/2} \opT^{\ast}$.
If $\lmf$ and $\opT^\ast\opT$ commute, then we will show that $\Ck$ coincides with  the
singular projection of $\lmg$, and we can bound the SPC as in the
  native case.

In the non-commuting case we cannot ensure that
$\Ck$ is the singular projection of $\lmg$. We assign the intrinsic
mapping~$H:= \opT^{\ast}\opT$, and work under Assumption
\ref{ass:prior-linked2scale-noncomm}, linking the operators $\lmf$ and
$\opT$ via~$H$. Notice that our treatment in Section
  \ref{sec:direct} is standalone and does not necessarily correspond
  to an inverse problem, however with a slight abuse of notation we can
  let~$H= \opT^{\ast}\opT$ and assume the link condition described in Assumption
\ref{ass:prior-linked2scale-noncomm}. 

Within  this (finite dimensional) Gaussian-Gaussian conjugate setting, given the centered Gaussian prior with covariance~$\Ck$, the posterior is also Gaussian with  mean and covariance
\begin{align}
  \hat g_{k} &= \lr{\Ck + \frac 1 n}^{-1} \Ck Y^{n},\label{ali:mean-ck}\\
  C_{post,k} &= \frac 1 n  \lr{\Ck + \frac 1 n}^{-1}\Ck.\label{ali:cov-ck}
\end{align}
In alignment with~\cite[Eq.~(3)]{MR3815105}, for any given~$g_0$ and truncation level~$k$, the~$\spc$ is decomposed as 
\begin{equation}
  \label{eq:spc-sum}
  \spc(k,g_{0}) = \norm{ g_{0} - \E\hat g_{k}}{Y}^{2} +  \frac 1 n \E\norm{\lr{\Ck + \frac 1 n}^{-1} \Ck\xi}{}^{2} + \tr\left[C_{post,k}\right],
\end{equation}
where the first summand is the (squared) bias for estimating~$g_{0}$
by using the posterior mean~$\hat g_{k}$, the second summand is the related estimation
variance, whereas the last summand constitutes the posterior spread.
The proof of the next result is based on this decomposition.

\begin{prop}
  \label{prop:error-direct-unified}
 Consider the white noise model~(\ref{eq:noise-model}) with
a Gaussian prior~$\mathcal N(0,  C_{k})$, and with underlying
truth~$g_{0}\in\lpsi$ for some index function~$\psi$ (see~(\ref{eq:lpsi})), where either
\begin{enumerate}
    \item\label{it:case1-prop} $\Ck=Q_k\lmg Q_k$ (native
    ), or   
    \item\label{it:case2-prop} $\Ck=\opT P_k\lmf P_k\opT^\ast$ (inherited with non-commuting $\lmf, \opT^\ast\opT$, linked via Assumption
\ref{ass:prior-linked2scale-noncomm} for $H=\opT^\ast\opT$), 
\end{enumerate}
  and where in the latter case we let~$\lmg=\opT \lmf \opT^\ast$, and we assume that  the function~$\psi$ is operator concave.
There is a constant~$\ca
 \geq 2$ such that
 for any truncation level~$k$  the squared posterior contraction is bounded as 
   \begin{equation}
   \label{eq:spc-bound}
      \spc^{\lpsi}(k)
      \leq \ca\lr{\psi^{2}{\lr{\frac 1 { n}}}+
      \psi^{2}\lr{s_{k+1}(\lmg)} + \frac{k}{n}}.      
    \end{equation}
\end{prop}

\subsection{Optimized SPC bound}
\label{sec:optimized-spc}
We aim at optimizing the general
bound~\eqref{eq:spc-bound}. This bound is constituted of two $k$-dependent terms, and a summand
$$
  \max\set{\psi^{2}\lr{\frac 1 {n}}, \frac 1 {n^{2}}},  
$$
which is independent of the truncation level~$k$.
As can be seen in the proof of Proposition~\ref{prop:error-direct-unified}, namely~\eqref{eq:spc-bound-2terms}, this summand is the result of bounding the \emph{regularization
   bias}, inherent in Bayesian problems with (untruncated) Gaussian
 priors. Hence the best (provable, by bounding the $\spc$ as above)
 contraction rate, will be bounded below (in order) by this
 regularization bias.

To better understand the nature of the $k$-dependent terms in the bound~\eqref{eq:spc-bound} we recall  the following result from statistical inference. The minimax risk over the class~$\lpsi$  is given as
$$
R(n):= \inf_{\hat g}\sup_{g\in \lpsi} \mathbb E\norm{g - \hat g}{}^2,
$$
where the infimum runs over all estimators using data~$Y^n$.
Similarly, let
$$
R_T(n):= \inf_{\hat g_T}\sup_{g\in \lpsi} \mathbb E\norm{g - \hat g_T}{}^2,
$$
where the infimum is taken over all (linear) truncated series estimators.
Since the class $\lpsi$ constitutes an ellipsoid, the following result holds.
\begin{prop}[{\cite[Prop.~8]{MR1062717}}]\label{prop:donohoetal}
We have that
$$
R_T(n) \leq 2.22 R(n).
$$
In particular
$$
R_T(n) = \inf_{k} \lr{\psi^{2}\lr{s_{k+1}(\lmg)} +k/n} \leq 2.22 R(n). 
$$
\end{prop}
We are ready to optimize the bound established in
Proposition~\ref{prop:error-direct-unified},
while Proposition \ref{prop:donohoetal} will enable the comparison of our optimized bound to the minimax rate.
\begin{thm}\label{thm:spc-bound}
Consider the white noise model~(\ref{eq:noise-model}) with a Gaussian prior~$\mathcal N(0,\Ck)$, and with underlying truth~$g_0\in\lpsi$ as in~(\ref{eq:lpsi}), where either 
\begin{enumerate}
    \item\label{it:case1} $\Ck=Q_k\lmg Q_k$ (native
    ), or   
    \item\label{it:case2} $\Ck=\opT P_k\lmf P_k\opT^\ast$ (inherited with non-commuting $\lmf, \opT^\ast\opT$, linked via Assumption
\ref{ass:prior-linked2scale-noncomm} for $H=\opT^\ast\opT$), 
\end{enumerate}
and where in the latter case we let~$\lmg=\opT \lmf \opT^\ast$, and we
assume that the function~$\psi$ is operator
concave.  We assign $k=\kn$ as in~(\ref{eq:kn-def}).
Then, for the constant~$\ca$ from Proposition~\ref{prop:error-direct-unified}, we have
\begin{equation}\label{eq:spc-bound-thm}
  \spc^{\lpsi}(\kn) \leq 4\ca \max\set{\psi^{2}\lr{\frac 1 n},\frac {\kn} n}.
\end{equation}
If the regularization bias in~\eqref{eq:spc-bound-thm} is of lower order, then the obtained contraction rate over the class $\lmg_\psi$ is order optimal.
\end{thm}

\begin{rem}\label{rem:kn-infinite}
    We emphasize that necessarily~$\kn\to\infty$ as~$n\to\infty$, because otherwise, if~$\kn < K<\infty$, then, from~\eqref{eq:kn-def} we find that
    $$
    \psi^2(s_{K}(\lmg))\leq \psi^2(s_{\kn +1}(\lmg)) \leq
    \max\set{\psi^2(1/n), \frac{\kn +1}{n}}\leq \max\set{\psi^2(1/n), \frac{K}{n}}\to 0,
    $$
    hence by the properties of index functions we have~$s_{K}(\lmg)=0$,  which is a contradiction.
\end{rem}

\begin{rem}\label{rem:truncation-dominating}
  The case that~$\kn/n < \psi^2(1/n)$ corresponds to the situation when the regularization bias dominates the overall~$\spc$. In this case the truncation level is obtained from the relation~$\kn = \max\set{j, \quad s_{j}(\lmg) > 1/n }$, and this may be significantly smaller than the truncation level obtained in the case that the regularization bias is dominated. 
\end{rem}

It is thus interesting to characterize those cases when the regularization bias
in~\eqref{eq:spc-bound-thm} is of lower order. We shall provide a characterization; but for this we need an additional assumption.
\begin{ass}[control of decay of singular numbers]\label{ass:decay-rate}
  There is a constant~$\cc~>~1$ with
\begin{equation}
  \label{eq:decay-rate}
\sup_{j\in\nat}\frac{s_{j}(\lmg)}{s_{j+1}(\lmg)} \leq \cc.
\end{equation}
\end{ass}
This assumption does not hold for operators $\lmg$ with singular values decaying faster than exponentially.
  \begin{prop}\label{prop:alphalessbeta}
    Under Assumption~\ref{ass:decay-rate}, the term~$\kn/n$ in~\eqref{eq:spc-bound-thm} dominates the overall bound for the $\spc$ if and only if there is a constant~$1 \leq\cd<\infty$ such that
  \begin{equation}
    \label{eq:alphabeta}
  \psi^{2}\lr{s_{j}(\lmg)} \leq \cd j s_{j}(\lmg),\quad j\in\nat.   
  \end{equation}
\end{prop}

Specifically, Proposition~\ref{prop:alphalessbeta} 
applies for problems with covariance operator $\lmg$ of the underlying infinite dimensional prior on $g$, having a power type decay of the singular numbers. Thus, in such cases, under Assumption~\ref{ass:decay-rate} the truncation level~$\kn$ yields order optimal contraction exactly if~\eqref{eq:alphabeta} holds.
\begin{xmplno}[{$\alpha$-regular prior and Sobolev smoothness}]
  For a native~$\alpha$-regular prior defined in $Y$ (recall the example in~\S~\ref{sec:contraction}), with (untruncated) covariance operator~$\lmg$, the Sobolev type smoothness of the
  underlying truth~$g_{0}\in \eS {\beta}$ is expressed through the
  index function~$\psi(t)= t^{\beta/(1 + 2\alpha)},\ t>0$ (recall the example in~\S~\ref{sec:smoothness}). For this
  function we see that for~$\alpha \leq \beta$ it holds that
  $$
  \psi^{2}(s_{j}(\lmg)) \asymp j^{-2\beta} \,\lesssim\,  j \cdot j^{-(1 + 2 \alpha)},
  $$
  such that both, Assumption \ref{ass:decay-rate} and~condition~(\ref{eq:alphabeta}) hold, and Proposition~\ref{prop:alphalessbeta} applies.

  The truncation level~$\kn$ is then given from balancing~$\frac k n
  \asymp k^{-2\beta}$, which results in~$\kn \asymp n^{1/(1 +
    2\beta)}$, yielding a bound for the~$\spc$ of the form
  $$
  \spc^{\lpsi}(\kn) \,\lesssim\, \frac{\kn}{n} \asymp n^{- \frac{2\beta}{1 + 2\beta}},
  $$
  which is known to be minimax for direct estimation.
  The same bounds are valid also for inherited $\alpha$-regular priors with commuting operators $\lmf, \opT^\ast\opT$. In the non-commuting case, provided $H=\opT^\ast\opT$ and $\lmf$ satisfy Assumption \ref{ass:prior-linked2scale-noncomm}, the same bounds hold for $\alpha\leq\beta\leq1+2\alpha$, where the additional restriction on $\beta$ is needed in order to ensure that $\psi$ is operator concave.
\end{xmplno}

\subsection{Interlude}
\label{sec:interlude}

Frequentist convergence rates of the posterior distribution under Gaussian priors in the Gaussian white noise
model, have been considered for example in~\cite{MR1790008} (rates for the posterior mean under Sobolev-type smoothness),~\cite{MR1983541} (contraction rates under Sobolev-type smoothness), and~\cite{MR2418663} (general contraction theory). We gave a
detailed discussion here on the one hand because, as explained in Section~\ref{sec:outline}, we are interested in general smoothness assumptions, and on the other hand because we want to emphasize the specifics when using
truncated (Gaussian) priors.


Theorem~\ref{thm:spc-bound} highlights the general nature of our bounds
for the squared posterior contraction ($\spc$), in terms of both the considered prior covariances, and the smoothness of the truth, expressed using source sets. 
In our analysis we distinguish two cases: case~\ref{it:case1}, which uses native priors, and which is entirely based on the singular value decomposition of the underlying covariance operator, and case~\ref{it:case2}, which refers to priors inherited from external native priors using some linear mapping, and which is such that the inherited finite dimensional prior is no longer supported in a singular subspace of the covariance operator of the underlying infinite dimensional inherited prior. The latter case, which is relevant when studying the direct problem \eqref{eq:direct} associated to the inverse problem~\eqref{eq:Ys} can be treated provided that the linear mapping is appropriately linked to the external native prior's covariance. In particular this link, captured in Assumption~\ref{ass:prior-linked2scale-noncomm}, imposes a minimum smoothness on the external native prior.

Special
emphasis is put on the description of the optimal truncation
level~$\kn$, made explicit in~(\ref{eq:kn-def}). It is seen that in
general this
level will depend on the underlying smoothness as well as on
the noise level~$\frac 1 {\sqrt n}$,  and that, in the case that the regularization bias is dominated, it is the
same as the truncation level in (minimax) statistical estimation under white
noise when using truncated series estimators, as expressed in
Proposition~\ref{prop:donohoetal}. Furthermore, the obtained upper
bound on the contraction rate involves a truncation-independent term,
the regularization bias, and thus in Proposition~\ref{prop:alphalessbeta} we give a characterization to determine
whether this term will be of lower order compared to the $k$-dependent
terms, or it will be dominating. In the former case the obtained rates of contraction are minimax, while in the latter case they are suboptimal\footnote{We stress that here suboptimality refers to the \emph{obtained} rates, which are upper bounds for the rate of contraction. Lower bounds can in principle be obtained using the theory developed in \cite{IC08}, which is based on the concentration function of the Gaussian prior rather than bounding the SPC. Although very interesting, we presently do not pursue this direction.}.  As already mentioned in Remark~\ref{rem:truncation-dominating}, the truncation level according to~\eqref{eq:kn-def} will be smaller for dominating regularization bias than for the case when the regularization bias is dominated.

We close this discussion with the following observation. In studies
dealing with scaled infinite-dimensional Gaussian priors a typical
'saturation effect' is observed: In order to achieve minimax-optimal rates of contraction the prior smoothness must not be much
lower than the regularity of the underlying truth, see~\cite{MR2906881} and~\cite{MR3044507}. The contrary is
true for truncated priors: when applying
Proposition~\ref{prop:alphalessbeta} in specific examples later on, it
will be transparent that the prior regularity must be lower than the
regularity of the underlying truth; see the preceding example as well. This has also been observed in~\cite{MR2418663}, and it can be explained by the fact that truncation of a Gaussian prior increases
its regularity, which can correct for an under-smoothing but not for an over-smoothing prior.
In the case that the truncation of the Gaussian prior is not along some singular subspace, a limitation for the considered smoothness occurs due to the nature of the linking Assumption~\ref{ass:prior-linked2scale-noncomm}. This can be seen from the final example in~\S~\ref{sec:optimized-spc}.\


\section{Modulus of continuity for inverse problems}
\label{sec:inverseP}
We next consider the linear mapping~$\opA\colon X\to Y$ from~\eqref{eq:Ys} and shall introduce the modulus of continuity for controlling its inversion on a subset $S$ (often called \emph{conditional stability}). We
shall do this for $S:= \Sk$, where~$\Sk\subset X$ is a $k$-dimensional
subspace. 
We derive bounds on the modulus of continuity which are known to be sharp in many cases.  

\subsection{Modulus of continuity}
\label{sec:modulus}

Similarly to the recent study~\cite{MR3757524},  but restricting to
normed spaces,  we proceed as follows.
Given the operator $\opA$, for a class~$S \subset X$  and a fixed element $f_0\in X$, we let
\begin{equation}
  \label{eq:mod-k-1}
 \omf(\opA^{-1},S,\delta):= \sup\set{\norm{f - f_{0}}{X},\ f\in S,\
    \norm{\opA(f -f_{0})}{Y}\leq \delta},\ \delta>0,
\end{equation}
be the modulus of continuity function.
We stress that this modulus function controls the deviation around the
 element~$f_{0}$, and hence it is local.

Recall the self-adjoint companion of $\opA$ introduced in §~\ref{sec:relating-ops}, $H=\asta$. It is evident that
\begin{align}
    \omf(\opA^{-1},S,\delta) &=  \omf(H^{-1/2}, S,\delta ),\quad \delta>0,\label{eq:omega-A-H}
  \end{align}
hence we shall confine the subsequent analysis to the operator~$H$.

\subsection{Bounding the modulus of continuity}
\label{sec:main-bound}

When bounding the modulus of continuity for the inversion of an operator around an element $f_0\in X$, it is convenient to express the smoothness of $f_0$ \emph{relative to that particular operator}. Precisely, in the context of the inverse problem~(\ref{eq:Ys}), we
shall measure the smoothness \emph{relative to the operator~$H$}, the companion of~$\opA$, and we shall assume that~$f_{0}\in\af$ for some index function~$\varphi$, see~§~\ref{sec:smoothness} where source conditions were introduced.


The control of the modulus of continuity is based on several
assumptions, relating a finite dimensional subspace~$\Sk\subset X$ to
the operator~$H$ as well as to the target function~$f_{0}$. We denote by $P_k$ the orthogonal projection of $X$ onto the subspace $\Sk$. Furthermore, recall that we denote by~$s_{k}:= s_{k}(H)$, the $k$-th
singular number of the (compact) operator~$H$.

\begin{de}
  [degree of approximation]\label{de:degree-approximation}
Let~$\opK\colon X\to Y$ be a (compact) operator.
Given a finite dimensional subspace~$\Sk\subset X$ we assign
$$
\varrho(\opK,\Sk):= \norm{\opK(I - P_{k})}{X\to Y}
$$
the degree of approximation of the subspace~$\Sk$ for the
operator~$\opK$. 
\end{de}

\begin{de}
  [modulus of injectivity]\label{de:modulus-injectivity}
Let~$\opK\colon X\to Y$ be a (compact) operator.
Given a finite dimensional subspace~$\Sk\subset X$ we assign
$$
j(\opK,\Sk) := \inf_{h\in
  \Sk\setminus\{0\}}\frac{\norm{\opK h }{Y}}{\norm{h}{X}},
$$
the modulus of injectivity, which quantifies the invertibility of the operator $\opK$ on the subspace $K(\Sk)$.
\end{de}
We mention here, that the last two concepts are interesting for sequences of increasing subspaces $\Sk$. 
Taking $K=H^{1/2}:X\to X$, the quantities $\varrho(H^{1/2},\Sk)$ and~$j(H^{1/2},\Sk)$ shall allow us to quantify the impact of the choice $S=\Sk$, when bounding the modulus of continuity $\omf(H^{-1/2}, S,\delta )$.
    
\begin{rem}   
    The above~$\varrho(H^{1/2},\Sk)$ relates to the \emph{$k$-th
      Kolmogorov number}, 
    while~$j(H^{1/2},\Sk)$ relates to the \emph{$k$-th Bernstein
      number}, both of which are well studied quantities in approximation theory, see~\cite{MR774404}. 
  \end{rem}

 When $\Sk$ is the $k$-th singular subspace of $H$, then it can be seen that $\varrho(H^{1/2},\Sk)=s_{k+1}^{1/2}$, $j(H^{1/2},\Sk)=s_{k}^{1/2}$ and $\norm{(I-P_{k})\varphi(H)}{}=~ \varphi(s_{k+1}),$ for any index function $\varphi$. When using a subspace~$\Sk$ other than the $k$-th singular
  space, then its quality with respect to the $k$-th singular subspace
  is measured in terms of Jackson and Bernstein inequalities which
  look as follows. 
  \begin{ass}
    [{relating~$\Sk$ to the $k$-th singular
      subspace of $H^{1/2}$}]\label{ass:relations}
Consider a sequence $(\Sk)_{k\in\mathbb N}$ of subspaces of $X$. There are constants~$M, C_{P},C_{B}\ge 1$ such that for $k\in\nat$ we have
  \begin{align}
    \varrho(H^{1/2},\Sk) & \leq C_{P} s_{k+1}^{1/2},&\qquad  \text{(Jackson Inequality)}\\
j(H^{1/2},\Sk) & \geq \frac 1 {C_{B}}s_{k}^{1/2},&\qquad \text{(Bernstein
                                                     Inequality)}
    \intertext{and for~$f_{0}\in\af$ we have that } 
    \norm{(I - P_{k})f_{0}}{X} &\leq M \varphi(s_{k+1}).& \qquad
    \text{(approximation power)}
  \end{align}
  \end{ass}
  \begin{rem}
Within the context of
projection schemes in classical ill-posed problems such assumptions were made in the study~\cite{MR2394505}. For finite element approximations, i.e.,\ when the spaces~$\Sk$ consist of finite elements, a detailed example is given in~\cite[Ex.~2.4]{MR2367863}.
In the context of Bayesian methods the recent study~\cite{MR4116718} also makes similar assumptions, see ibid. Ass.~2.3.
  \end{rem} 

  Under Assumption~\ref{ass:relations} the following bound holds.
  \begin{prop}\label{prop:main-bound}
  Let $f_0\in \af$ and let the sequence $(\Sk)_{k\in\mathbb N}$ of subspaces of $X$ satisfy Assumption~\ref{ass:relations}. Then for all $k\in\mathbb N$, we have that
$$
\omega_{f_{0}}(H^{-1/2},\Sk,\delta) \leq M\lr{1 +C_{P}C_{B} 
}\varphi(s_{k+1}) + C_{B}
\frac{\delta}{\sqrt{s_{k}}}.
$$
\end{prop}


In the bound from Proposition~\ref{prop:main-bound} we have the 
flexibility of choosing the truncation level~$k\in\nat$, and we next study this choice. 
First, we recall
the following companion to the index function~$\varphi$  as
  \begin{equation}
    \label{eq:Theta}
    \Theta_{\varphi}(t) := \sqrt t \varphi(t),\quad t>0.
  \end{equation}
  Notice that $\Theta_\varphi$ is also an index function, more
  specifically, it is always strictly increasing hence invertible.

  Optimizing the bound from Proposition~\ref{prop:main-bound}
    with respect to the choice of the truncation level, we arrive at the main result of this section.
\begin{thm}
  \label{thm:phi-theta-bound}
  Suppose that~$f_{0}\in\af$, and that $(\Sk)_{k\in\mathbb N}$
 satisfies  Assumption~\ref{ass:relations}. Given~$\delta>0$ we   assign
 \begin{equation}
   \label{eq:nast}
   \nast := \max\set{j,\quad \Theta_{\varphi}(s_{j}) > \delta}.
 \end{equation}
  Then there is a constant~$\cg$ such that 
\begin{equation}\label{eq:modbound}
\omega_{f_{0}}(H^{-1/2},\Skd,\delta) \leq \cg \varphi\lr{\Theta_{\varphi}^{-1}(\delta)},
\end{equation}
for~$\delta>0$ small enough.
\end{thm}

Some extensions to the above bounds on the modulus of continuity, can be found in Appendix \ref{app:A}.


\section{Relating the contraction rates for the direct and inverse
  problems}
\label{sec:relating}

In this section we discuss the steps for proving Theorem~\ref{thm:direct-inverse}, which is an application of Proposition~\ref{prop:ks}.
  
We shall first use Theorem~\ref{thm:spc-bound}
to establish contraction rates for the direct problem~\eqref{eq:direct}, finding rate sequences~$\delta_n,$ for truncation level $\kn$, $ n\in\nat$. 
 In order to apply Theorem~\ref{thm:spc-bound} we need to determine
  the inherited prior for the direct problem (formulated in $Y$), obtained by pushing forward the (truncated Gaussian) prior on~$f$ through the mapping~$\opA$ (formulated in $X$). Furthermore, given an element~$f_0\in X$ we need
  to express the smoothness of~$g_0= \opA f_{0}$ with respect to the
  corresponding (inherited, untruncated) covariance operator.
We address both of these tasks in §~\ref{sec:relating-problems} and derive rates $\delta_n$ for the direct problem, by relying on either Assumption~\ref{ass:link-noncomm} or~\ref{ass:prior-linked2scale-noncomm}, depending on whether the (untruncated) prior covariance on $f$ commutes with $H=\asta$ or not.

Given such a rate $\delta_n$, we can then use the results of Section \ref{sec:inverseP}, specifically Proposition \ref{prop:main-bound}, to compute the corresponding $\omega_{f_0}(H^{-\frac12},\Xkn,\delta_n)$, which according Proposition~\ref{prop:ks}, is a rate of contraction for the inverse problem at $f_0$. Here, $\lr{X_k}_{k\in\nat}$ are the singular spaces of the prior covariance operator~$\lmf$. 
 A main component of the proof will be the realization that $\kn$ as given in \eqref{eq:kn-def}, optimizes both the contraction rate $\delta_n$ as well as our bounds on the modulus of continuity; we shall see this in 
§~\ref{sec:optimality}. In the course, we shall establish that $\lr{X_k}_{k\in\nat}$ obey Assumption~\ref{ass:relations}. 
\subsection{Rates for the direct problem}
\label{sec:relating-problems}

Let us consider the model~\eqref{eq:Ys}, and put a Gaussian prior $\Pf=\mathcal{N}(0, P_k \lmf P_k)$ on $f\in X$, for a given self-adjoint, positive-definite and trace-class operator $\lmf:X\to X$. Here $P_{k}$ denotes the orthogonal projection onto the $k$-dimensional subspace $\Xk\subset X$ corresponding to the singular value decomposition of the operator $\lmf$. We are interested in finding contraction rates $\delta_n$ of~$\opA f$ 
around~$\opA f_{0}$, for a given $f_{0}\in X$. 

Due to linearity, the Gaussian prior on $f\in X$, induces a Gaussian
prior $\Pg$ on $\opA f\in Y$, which has zero mean and covariance
operator
\begin{equation}
  \label{eq:cov-push-farward}
  \Ck= \opA  P_{k} \lmf P_{k}\opA^{\ast}.
\end{equation}

Recall the terminology of native and inherited priors from Section \ref{sec:direct}. It is interesting to ask when this push-forward prior is native for~$g\in
Y$ and this is the case when the operators~$H$ and~$\lmf$
commute.
However, in general this will not be the case, that is, the push-forward
of~$\Pf$ will not be native in~$Y$.
Nevertheless, the $\spc$ was bounded in~(\ref{ali:mean-ck}) for both
native and non-native priors, respectively. See also Theorem \ref{thm:spc-bound}, which optimizes the bounds in both cases. 
\subsubsection{Commuting case: general smoothness}
\label{sec:commute}

The main observation is comprised as follows.
\begin{lem}
  \label{lem:commuting-H-lmf}
  If the operators~$H=\asta$ and~$\lmf$ commute, then the push-forward
  prior~$\opA_{\sharp \lr{\Pf}}$ is native for~$g$.

  Furthermore, under Assumption~\ref{ass:link-noncomm} (with index
  function~$\chi$) we see that
  \begin{enumerate}
  \item the covariance operator~$\lmg = \opA \lmf \opA^{\ast}$ has a representation
    \begin{equation}
    \label{eq:lmg}
   \lmg= \Theta_{\chi}^{2}(\UU H\UU^\ast),
 \end{equation}
 where~$U$ is an isometry arising from the polar decomposition~$A = U
 H^{1/2}$, and
 \item a source conditon~$f_{0}\in\af$ yields that~$g_{0}\in\lpsi$ with index function~$\psi(t):=
\Theta_{\varphi}(\lr{\Theta_{\chi}^{2}}^{-1}(t)),\ t>0$.
  \end{enumerate}

\end{lem}
Based on the above technical result we state the following consequence.
\begin{prop}\label{prop:relate-direct}
Let $\lmf:X\to X$ be a positive definite, self-adjoint, trace class linear operator, and consider the companion $H=\asta$ to the forward operator $\opA:~X~\to~Y.$  Under Assumption \ref{ass:link-noncomm}, assign to any index function~$\varphi$ the related index function~$\psi(t): =\Theta_{\varphi}(\lr{\Theta_{\chi}^{2}}^{-1}(t)),\ t>0$.
  The following statements regarding a rate sequence~$\lr{\delta_n}_{n\in\nat}$ are equivalent:
  \begin{enumerate}
  \item[a)]$\lr{\delta_n}_{n\in\nat}$ is a contraction rate for the
    direct problem~\eqref{eq:direct} around $g_0=\opA f_0$, 
   under the push-forwards $\opA_{\sharp \lr{\Pftn}}$ of the sequence of Gaussian priors $\Pftn~=~\mathcal{N}(0,P_{\ktn}\lmf P_{\ktn})$ on $f$, where $P_k$ is the orthogonal projection onto the $k$-th singular space $\Xk$ of $\lmf$,  and for $f_0\in {\af}$.
 \item[b)]$\lr{\delta_n}_{n\in\nat}$ is a rate of contraction for
   model \eqref{eq:noise-model}, obtained for a sequence of (native) Gaussian priors~$\mathcal N(0, Q_{\ktn}\Theta_{\chi}^{2}(\UU H\UU^\ast)Q_{\ktn})$ on $g$, where $Q_k$ is the orthogonal projection onto the $k$-th singular space $\UU \Xk$ of $\lmg:=~\Theta_{\chi}^2(\UU H\UU^\ast)$, and for $g_0\in \lmg_\psi$.
  \end{enumerate}
\end{prop}

\subsubsection{Non-commuting case: power type smoothness}
\label{sec:non-commute}

If the operators $H=\asta$ and $\lmf$ do not commute, the push-forward of the prior on~$f$ will no longer be native for~$g= \opA f$. However, even in the
non-commuting case we can translate the smoothness assumption~$f_{0}\in \af$ with power-type $\varphi$, to a corresponding
smoothness of~$g_{0}:= \opA f_{0}$ with respect to the operator~$\lmg=\opA\lmf\opA^\ast$, under Assumption~\ref{ass:prior-linked2scale-noncomm}.
\begin{lem}\label{lem:smoothness-h2lmg}
  Suppose that Assumption~\ref{ass:prior-linked2scale-noncomm} holds. If~$f\in \af$ for
  an index function~$\varphi(t)\propto t^{\mu}$ with~$\mu\leq a$, then~$g_{0}= \opA
  f_{0} \in\lmg_{\psi}$ for an index function~$\psi(t) \propto~t^{\frac{\mu +
      1/2}{2a+1}}$, which has an operator concave square.
\end{lem}

The proof of Lemma~\ref{lem:smoothness-h2lmg}, which holds for
$\mu\leq a$, is based on Heinz' Inequality, and this allows to treat power-type smoothness of $g_{0}$ with respect to $\lmg$, with exponent $0\leq\theta\leq 1/2$. In particular, it does not allow to fully exploit the results of Section~\ref{sec:direct} for inherited priors, which hold for $0\leq\theta\leq1$ (since they only require operator concavity of $\psi$). 

Therefore, we shall highlight the following condition, which allows to extend the range of applicability in the non-commuting cases. It is a strengthening of Assumption~\ref{ass:prior-linked2scale-noncomm}:
There exists $a\geq1/2$ such that
\begin{equation}
  \label{eq:3over2}
  \norm{\lr{\lmf}^{3/2}f}{X} \asymp \norm{H^{3a}f}{X},\quad f\in X.
\end{equation}
\begin{rem}
In view of Heinz' Inequality (with~$\theta:= 1/3$), \eqref{eq:3over2} is consistent
with Assumption~\ref{ass:prior-linked2scale-noncomm}. Conversely, in this
non-commuting case, \eqref{eq:3over2} cannot be derived from Assumption~\ref{ass:prior-linked2scale-noncomm}, but instead is a strengthening of it. In brief, the validity of a link condition yields that the eigenfunctions for the operators on both sides must share the same smoothness (which can be seen from the modulus of injectivity, reflecting the 'inverse property'). Therefore, in general a link cannot be 'lifted' to higher powers, contrasting the commuting case, where both sides share the same eigenfunctions, and so do arbitrary powers.  
\end{rem}

The proof of the  following strengthening of
Lemma~\ref{lem:smoothness-h2lmg} is  based on
interpolation in scales of Hilbert spaces, a concept which
extends Heinz' Inequality which interpolates 'element-wise' to 'operator-wise' interpolation.
\begin{lem}\label{lem:smoothness-h2lmg-3over2}
  Suppose that~(\ref{eq:3over2}) holds. If~$f\in \af$ for
  the index function $\varphi(t)=~t^{\mu}$ with~$\mu\leq 2a + 1/2$, then~$g_{0}= \opA
  f_{0} \in\lmg_{\psi}$ for the index function~$\psi(t) = t^{\frac{\mu +
      1/2}{2a+1}}$, which is operator concave.
\end{lem}

We summarize the developments of this section.
\begin{prop}\label{prop:relate-direct-noncomm}
   Let $\lmf:X\to X$ be a positive definite, self-adjoint, trace class linear operator, and consider the companion $H=\asta$ to the forward operator $\opA:X\to Y.$ Consider a
  native Gaussian prior $\Pf=\mathcal{N}(0,P_{k}\lmf P_{k})$ for~$f\in X$, where $P_k$ is the orthogonal projection onto the $k$-th singular space $\Xk$ of $\lmf$.
  
  Given $a\geq 1/2$, assign
  to the index function~$\varphi(t) = t^{\mu}, \mu>0$, the related index
  function~$\psi(t): = t^{(\mu + 1/2)/(2a+1)},\ t>0$.
  Suppose either~$\mu \leq a$ and Assumption~\ref{ass:prior-linked2scale-noncomm} holds, or $\mu\leq 2a+1/2$ and \eqref{eq:3over2} holds.
  
 Consider the direct problem~\eqref{eq:direct} around $g_0=\opA f_0$, under the sequence of priors $\Pftn$ on $f$ and for $f_0\in {\af}$. Then we can obtain a rate of contraction for this problem, by computing a rate of contraction
$\lr{\delta_n}_{n\in\nat}$ for
   model \eqref{eq:noise-model}, for the sequence $\opA_{\sharp \lr{\Pftn}}$ of inherited Gaussian priors on $g$, and for $g_0\in \lmg_\psi$ with $\lmg=\opA\lmf\opA^\ast$.
\end{prop}

We conclude this discussion on relating the obtained smoothness of~$g_{0}=\opA f_{0}$ to the smoothness of $f_0$, as
 expressed in Propositions~\ref{prop:relate-direct}
 and~\ref{prop:relate-direct-noncomm}, respectively for the commuting and non-commuting cases.
Specifying $\chi(t):=t^a$ and $\varphi(t):=t^\mu$ in the commuting case, we restrict to the power-type smoothness and relationship between $\lmf$ and $H$, considered in the non-commuting case. In that setting, the obtained functions, representing the smoothness, should thus agree. Indeed, it is readily seen that the function~$\psi$ as obtained in Proposition~\ref{prop:relate-direct} is exactly the same as in
Proposition~\ref{prop:relate-direct-noncomm} with this specification. Therefore, the assumptions for the non-commuting case allow to maintain the results as obtained in the commuting one, however, the limitations~$a\geq 1/2$ and~$0< \mu \leq 2a+1/2$ occur, which are not seen in the commuting case.

\subsection{Rates for the inverse problem - optimality of the truncation point}
\label{sec:optimality}

Consider a forward operator $\opA$ and let $H:= \asta$ be its companion self-adjoint operator. Let $\delta_n$ be a
rate of contraction for the direct problem \eqref{eq:direct} around~$g_{0}=~\opA f_0\in Y,$ under a truncated at level $k_n$ Gaussian
prior as defined in the previous subsection. If $\lmf$ and $H$ commute, by Proposition
\ref{prop:relate-direct}, under Assumption~\ref{ass:link-noncomm} such a rate
can be computed using Theorem~\ref{thm:spc-bound}. Such a rate can also be computed
in the non-commuting case under
Assumption~\ref{ass:prior-linked2scale-noncomm},  and the corresponding result was
formulated in Proposition~\ref{prop:relate-direct-noncomm}.

Then according to
Proposition \ref{prop:ks}, in order to compute a rate of contraction
for the original inverse problem \eqref{eq:Ys}, it suffices to compute
$\veps_n=\omega_{f_0}(H^{-1/2},\Xkn,\delta_n)$. 

We have studied bounding the modulus of continuity $\omega_{f_0}(H^{-1/2},\Sk,\delta)$ in Section \ref{sec:inverseP}. Our bounds hold under Assumption \ref{ass:relations} on the relationship of the subspaces $(\Sk)_{k\in\mathbb N}$ to the singular subspaces of $H$. Since in the present Bayesian inverse problem context, $(\Xk)_{k\in\mathbb N}$ are aligned to the untruncated prior covariance operator $\lmf$, in order to apply the results of Section \ref{sec:inverseP} for bounding $\veps_n=\omega_{f_0}(H^{-1/2},\Xkn,\delta_n)$, we first need to verify that $(\Xk)_{k\in\mathbb N}$ satisfy Assumption \ref{ass:relations}. We do this in the next proposition.

  \begin{prop}\label{prop:validity-ass-relations}
    Let~$(\Xk)_{k\in\mathbb N}$ be the singular spaces for the operator~$\lmf$. 
Both Assumption~\ref{ass:link-noncomm}, or Assumption~\ref{ass:prior-linked2scale-noncomm} with smoothness~$f_{0}\in H_\varphi$ with $\varphi(t) = t^\mu$ for $0<\mu\leq a$, yield the validity of Assumption~\ref{ass:relations}. Under the stronger assumption~\eqref{eq:3over2} the range in the latter setting extends to~$\mu \leq 2a + 1/2$.
  \end{prop}
\begin{rem}
The above result is in correspondence with~\cite[Prop.~5.3]{MR4116718}, in which the commuting case is concerned. Here this is extended to the non-commuting cases under the link conditions (Assumption~\ref{ass:prior-linked2scale-noncomm} and~\eqref{eq:3over2}).
\end{rem}  

We next investigate whether the truncation level~$\kn$ from~\eqref{eq:kn-def},
also yields an optimized bound when used as a discretization level for
the modulus of continuity, such that both bounds are optimized
simultaneously. Indeed, we will see that this is the case and the following two technical results are the key. We first establish the optimality of $\kn$ in
the commuting case, and then extend to the
non-commuting case.

Given an index function~$\psi$ we consider a rate sequence~$\delta_n$ which obeys
\begin{equation}\label{eq:mild-ass}
  2 \max\set{\psi^2(1/n), \frac{\kn}{n}} \leq \delta_n^{2} \leq \ci^2 \max\set{\psi^2(1/n), \frac{\kn}{n}},
  \end{equation}
  for a constant~$2\leq\ci< \infty$.
  
  \begin{prop}\label{prop:relation}
  Under Assumption~\ref{ass:link-noncomm} the following holds true:
    suppose that $f_{0}\in H_\varphi$, and let~$\psi(t)=\Theta_{\varphi}(\lr{\Theta_{\chi}^{2}}^{-1}(t))$. Let~$\kn$ be as
  in~(\ref{eq:kn-def}), and assume that~\eqref{eq:mild-ass} holds true 
  for a rate sequence $\lr{\delta_n}_{n\in\nat}$.
  We then have that
  $$
  \omega_{f_0}\lr{H^{-1/2}, \Xkn,\delta_n} \leq  2
  (1 + \ci)\varphi\lr{\Theta_{\varphi}^{-1}(\delta_n)}.
  $$
\end{prop}
This result is extended to the non-commuting case as follows.
\begin{prop}\label{prop:relation-noncomm}
  Under Assumption~\ref{ass:prior-linked2scale-noncomm} with $\mu\leq a$, or Assumption~\eqref{eq:3over2} with $\mu\leq 2a+1/2$, the following holds true:
    suppose that~$f_{0}\in H_\varphi$ for the power type
    function~$\varphi(t) = t^{\mu}$, and let~$\psi(t)= t^{(\mu +
      1/2)/(2a +1)}$. Let~$\kn$ be as
  in~(\ref{eq:kn-def}), and assume that~\eqref{eq:mild-ass} holds true
  for a rate sequence $\lr{\delta_n}_{n\in\nat}$.
  We then have that
  $$
  \omega_{f_0}\lr{H^{-1/2}, \Xkn,\delta_n} \lesssim
  \delta_{n}^{\frac{\mu}{\mu +1/2}}.
  $$
\end{prop}

Evidently, for $\kn$ as in \eqref{eq:kn-def} a bound as
  in~\eqref{eq:mild-ass} holds for $\delta_n$ equal to the optimized
  bound for the direct problem as given in the right hand side of
  \eqref{eq:spc-bound-thm}, hence our bound on the modulus of
  continuity is indeed also optimized in both the commuting and non-commuting cases according to the last two propositions. 
  
  Combined, Propositions~\ref{prop:relation} and \ref{prop:relation-noncomm} imply the validity of Theorem~\ref{thm:direct-inverse}. We emphasize that Proposition~\ref{prop:relation-noncomm} holds true for the extended range~$\mu\leq 2a+1/2$, provided that condition~\eqref{eq:3over2} holds. This yields the following corollary.
  
  \begin{cor}[Corollary to the proof of Theorem~\ref{thm:direct-inverse}]
Consider the inverse problem \eqref{eq:Ys}, and
suppose that~$f_{0}$ has smoothness~$\af$ for the function~$\varphi(t) = t^\mu$. Assume we put a
truncated Gaussian prior $\mathcal{N}(0, P_{\kn}\lmf P_{\kn})$ on $f$,
with $\lmf$ a self-adjoint, positive-definite, trace-class, linear
operator in $X$, and specify the related covariance operator~$\lmg=A\lmf A^\ast$.

Under condition~\eqref{eq:3over2} with $\mu\leq 2a+1/2$,  consider the index function
\begin{equation*}
  \psi(t) =t^{\frac{\mu + 1/2}{2a+1}}, \quad t>0.
\end{equation*}

For the choice~$\kn$ according to~\eqref{eq:kn-def}
let $\delta_n$ be given as 
\begin{equation*}
   \delta_n:= C \max\set{\psi^{2}\lr{\frac 1 n},\frac {\kn} n} 
\end{equation*}
for some constant~$C$. 
Then the posterior contracts around $f_{0}$ at a rate 
  \[
  \veps_n\asymp \varphi(\Theta_\varphi^{-1}(\delta_n)),\quad n\to\infty.
  \]

\end{cor}


\section{Examples}
\label{sec:xmpls}

Here we exhibit how to use Theorem \ref{thm:direct-inverse} in order to obtain rates of contraction for the inverse problem~\eqref{eq:Ys}. The subsequent examples will distinguish between the decay of the singular numbers of the forward map~$\opA$, being moderate (power type), severe (exponential decay) or mild (logarithmic decay). Throughout we fix once and for all some element~$f_{0} \in\eS \beta$, see~\eqref{eq:sobolev-ell} in Section~\ref{xmpl:es-beta}. It will be transparent that, depending
on the underlying operator~$H=\asta$ this will result in different
source-wise representations~$f_0\in\af$. However, regardless of the kind of
ill-posedness of the operator~$H$ we will have that~$\varphi^2(s_j)\asymp j^{-2\beta}$.

For a truncated Gaussian prior on $f$ with underlying covariance operator $\lmf$, we thus need to determine SPC-rates~$\delta_n$ for the direct problems \eqref{eq:direct} which correspond to these examples. We will do this in Section~\ref{sec:direct-rates}, and we will apply Theorem~\ref{thm:spc-bound} which results in the bound \eqref{eq:spc-bound-thm} for the optimal truncation level $\kn$ given in \eqref{eq:kn-def}. For all considered types of behaviour of the singular numbers of~$\opA$, we will study truncated $\alpha$-regular Gaussian priors as introduced in Section~\ref{sec:contraction}. In addition, in the case that~$\opA$ exhibits exponential decay of the singular numbers, we shall also discuss a prior covariance operator with exponential decay (analytic prior), this is in alignment with the case analyzed in the study~\cite{MR3757524}. In all cases we will assume that $\lmf$ and $H$ commute.

Having determined rates~$\delta_n$ for the direct problem, in Section~\ref{sec:examples} we shall establish bounds for the modulus of continuity corresponding to the forward operators $\opA$ at hand. To this end we will apply Theorem \ref{thm:phi-theta-bound}, which for any $\delta$ results in the bound \eqref{eq:modbound} for the optimal truncation level $\nast$ given in \eqref{eq:nast}. We shall then highlight, that by Theorem \ref{thm:direct-inverse}, plugging $\delta=\delta_n$ in these bounds results in contraction rates for the corresponding inverse problems, for a truncated at level $\kn$ Gaussian prior. 

The rates given below for (most of) the direct and (all of the) inverse
  problems,  correspond to the minimax rates for
  estimation in Gaussian white noise, under Sobolev-type smoothness.  While for
  Examples~\ref{xmpl:power-type} and~\ref{xmpl:severe-type} these
  minimax rates are known, it is possible to find the minimax  rates for the
  mildly ill-posed case in Example~\ref{xmpl:mild-type}, by using
 Theorem~\ref{thm:spc-bound} for the direct problem and the result from~\cite{MR3859257} for the inverse problem. These rates are given here for the
  first time.

Finally, we will conclude with a discussion on non-commuting $\lmf$ and $H$ cases in  Section~\ref{sec:examples-discussion}.

\subsection{Direct rates}
\label{sec:direct-rates}

We confine to the case that $\lmf$ and $H$ commute, so that Assumption \ref{ass:link-noncomm} holds with appropriate~$\chi$. Recall that in this context, $\lmg=\opA\lmf\opA^\ast$, and that the smoothness of the truth is expressed relative to $\lmg$, via $\psi(t)$ given in \eqref{eq:psi-def}. Then, in order to obtain the truncation level~$\kn$ from~\eqref{eq:kn-def} and the corresponding bound on the $\spc$ from~\eqref{eq:spc-bound-thm}, we shall proceed as follows. 

In this commuting case we see that~$s_j(\lmg) = s_j(H) s_j(\lmf)$
and we first check if Assumption~\ref{ass:decay-rate} holds, in which case we can use Proposition~\ref{prop:alphalessbeta} to determine whether the regularization term dominates in the bound \eqref{eq:spc-bound-thm} or not. 
Furthermore, we make use of the identity 
\begin{math}
   \psi(\lmg) = \Theta_{\varphi}(\UU H \UU^\ast),
\end{math}
 which holds for  $\psi(t)$ from~\eqref{eq:psi-def}, and this extends to the singular
 numbers. 
 Using this identity,
 condition~(\ref{eq:alphabeta}) translates to
 \begin{equation}
   \label{eq:alpha-beta-new}
   (j^{-2\beta} \asymp)\quad  {\varphi}^{2}(s_{j}) \leq \cd j s_{j}(\lmf),\quad j=1,2,\dots
 \end{equation} 
Under Assumption~\ref{ass:decay-rate} and \eqref{eq:alpha-beta-new}, we find~$\kn$ by balancing~$k/n \asymp \psi^2\lr{s_k(\lmg)}$ and the~$\spc$ is bounded by (a multiple of)~$\kn/n$. This bound is known to be order optimal.

If Assumption~\ref{ass:decay-rate} does not hold, then we proceed as follows, cf. Remark~\ref{rem:truncation-dominating}. We find~$l_n$ by balancing~$l/n \asymp \psi^2\lr{s_l(\lmg)}$. 
Then we check whether~$\psi^2(1/n)$ is larger than~$l_n/n$, in which case the regularization bias dominates. If this is the case, then $\kn$ is found by balancing $s_{j}(\lmg)\asymp 1/n$ and the~$\spc$ is bounded by (a multiple of)~$\psi^2(1/n)$. Otherwise, $\kn=l_n$ and the $\spc$ is bounded by (a multiple of)~$\kn/n$. In the latter case this is known to be order optimal again. We emphasize that we only need to explicitly compute $\psi$ (hence also $\chi, (\Theta^2_\chi)^{-1}$), in the case that the regularization bias dominates.
Another consequence is worth mentioning. In case that the regularization bias is dominated, and hence the obtained contraction rate corresponds to the minimax rate of statistical estimation, then the truncation level~$\kn$ is obtained from balancing~$k/n \asymp \psi^2(s_k(\lmg)) = \Theta_\varphi^2(s_k(H))$. In particular, the level~$\kn$ does not depend on the chosen regularity of~$\lmg$, it is entirely determined by the smoothness as expressed with respect to~$H$. Similar applies to the contraction rate for the inverse problem. As the minimax rate cannot depend on the prior regularity the same holds for the chosen truncation level. This is seen in the examples, below.


Notice that in Example \ref{xmpl:severe-type} (both with $\alpha$-regular and analytic priors as considered below), the direct problem corresponds to a prior covariance and smoothness of the truth, which are not standard in the literature for the white noise model. Here they appear naturally, because the structure of the direct problem is inherited from the considered inverse problem. For this reason, it was necessary to have the general setup for the direct problem in Section \ref{sec:direct}.

\begin{xmpl}[moderately ill-posed operator]\label{xmpl:power-type}
Here we assume that the operator~$H$ has
power type decay of the singular numbers, that is,\ $s_{j}(H)\asymp
j^{-2p},\ p>0,\ j=1,2,\dots$. We need to find a corresponding index function such
that~$f_{0}\in \af$. This is achieved
by letting~$\varphi(t) := t^{\beta/(2p)}$, see the example in §~\ref{sec:smoothness}, which gives $\Theta_{\varphi}(t)=t^\frac{\beta+p}{2p}$. 
  We consider truncated $\alpha$-regular priors, so that $s_j(\lmf)\asymp j^{-1-2\alpha}.$
  
 Note that $g_{0}$ has smoothness $\lmg_\psi=(UHU^\ast)_{\Theta_\varphi}$, which in this example translates to Sobolev-type smoothness of order $\beta+p$.
  
We have $s_j(\lmg)=s_js_j(\lmf)\asymp j^{-1 - 2(\alpha+p)}$, and hence the regularity of the prior increases from~$\alpha$ to~$\alpha + p$, also.  Assumption~\ref{ass:decay-rate} holds in this case. For $\alpha$-regular priors
  condition~\eqref{eq:alpha-beta-new} holds  if and only
  if~$\alpha \leq \beta$, and in this case we  know from Proposition~\ref{prop:alphalessbeta} that the regularization bias in Theorem \ref{thm:spc-bound} is of lower order.
  The optimized truncation level $\kn$ as given in \eqref{eq:kn-def}, can thus be computed by balancing
  $$
  \frac k n \asymp \psi^{2}\lr{s_{k}(\lmg)} =
  \Theta_{\varphi}^{2}\lr{s_{k}}\asymp k^{-2(\beta + p)},
  $$
  yielding~$\kn \asymp n^{\frac1{1+2\beta + 2p}}$. Plugging this into \eqref{eq:spc-bound-thm}, we obtain the bound
  \begin{equation}\label{eq:deltan-ex1}
       \delta_n^2:=\spc^{\lpsi}(\kn)\lesssim n^{-\frac{2\beta+2p}{1+2\beta+2p}},
  \end{equation}
  which is the square of the minimax rate for the white noise model under Sobolev-type smoothness of order $\beta+p$ (this is both asserted by Theorem \ref{thm:spc-bound} but also well known in this case).  
\end{xmpl}

\begin{xmpl}[severely ill-posed operator]\label{xmpl:severe-type}  Here we assume that the operator~$H$ has
exponential decay of the singular numbers, that is,\ $s_{j}(H)\asymp e^{-
2 \gamma j^{p}}, \ p>0,\ j=1,2,\dots$. The resulting index function~$\varphi$
which realizes the source condition for~$f_{0}$ is then~$\varphi(t) =
\log^{-\beta/p}(1/t)$, and the related function~$\Theta_{\varphi}$ is
given as~$\Theta_{\varphi}(t)= \sqrt t \log^{-\beta/p}(1/t)$.
Lemma~\ref{lem:geometry} shows that its inverse behaves like $\Theta_{\varphi}^{-1}(s) \sim
s^{2}\log^{2\beta/p}(1/s)$. We again consider truncated $\alpha$-regular priors, so that $s_j(\lmf)\asymp j^{-1-2\alpha}.$

  
  Note that again~$g_{0}$ has smoothness $\lmg_\psi=(UHU^\ast)_{\Theta_\varphi}$, which in this example means that $g_{0}$ has coefficients decaying at least as fast as $e^{-\gamma j^p}/j^{\beta}$.
  
  We have that
  $s_j\lr{\lmg}=s_js_j(\lmf)\asymp j^{-1-2\alpha}e^{-2\gamma j^p}$. In this case,
  Assumption~\ref{ass:decay-rate} only holds if $p\leq1$. Since we are interested in all $p>0$, we
  cannot apply Proposition \ref{prop:alphalessbeta} and we need to
  check which of the two terms dominates in the bound
  \eqref{eq:spc-bound-thm} in Theorem~\ref{thm:spc-bound}, thus we compute $\psi(t)=\Theta_\varphi((\Theta_\chi^2)^{-1}(t))$, explicitly.
  
By Assumption~\ref{ass:link-noncomm} we have that $\chi^2(s_j)=s_j(\lmf)$. Thus~$\chi^2(t)\asymp \log^{-\frac{1+2\alpha}p}(1/t)$, and hence
  $\Theta^2_{\chi}(t)\asymp t\log^{-\frac{1+2\alpha}p}(1/t)$. 
 Using Lemma
  \ref{lem:geometry}, we can invert $\Theta^2_{\chi}$ to get
  $$(\Theta^2_\chi)^{-1}(s)\sim s\log^{\frac{1+2\alpha}p}(1/s)\quad \text{as}\ s\to0,$$  thus
  \[\psi(t)\asymp t^\frac12\log^{\frac{1+2\alpha-2\beta}{2p}}(1/t),\quad \text{as } t\to0.\]
  
  On the one hand the regularization bias behaves asymptotically as \[\psi^2(1/n)\asymp\frac1n\log^{\frac{1+2\alpha-2\beta}p}(n).\]
  
 On the other hand, we find $l_n$ by balancing 
$$
\frac l n \asymp \psi^{2}\lr{s_{l}(\lmg)} =
\Theta_{\varphi}^{2}\lr{s_{l}}\asymp l^{-2\beta}e^{-2\gamma l^p},
$$
 resulting in~$l_n\asymp\log^\frac1p(n)$, again using Lemma \ref{lem:geometry}.
 We thus see that the regularization bias is of lower order, i.e., $\psi^2(1/n)\lesssim l_n/n$, if and only if $\alpha\leq \beta$, in which case $\kn$ in \eqref{eq:kn-def} is equal to $l_n$.  
 For $\alpha>\beta$, the level $\kn$ can be found by balancing $s_k(\lmg)\asymp 1/n$, yielding $\kn\asymp \log^\frac1p(n)$ again.
 The right hand side of the bound \eqref{eq:spc-bound-thm} is dominated by $\kn/n\asymp \frac1n\log^{\frac1p}(n)$ in the former case, while in the latter by $\psi^2(1/n)$ as given above.

Combining, Theorem \ref{thm:spc-bound} gives the bound
\begin{equation}\label{eq:deltan-ex2}
\delta_n^2:=\spc^{\lpsi}(k_n)\lesssim 
n^{-1}\log^\frac{1+0\vee(2\alpha-2\beta)}p(n),
\end{equation}
which, whenever~$\alpha \leq \beta$, is the square of the minimax rate for the white noise model under the smoothness class $\lmg_\psi=(\UU H \UU^\ast)_{\Theta_\varphi}$ (this is both asserted by Theorem \ref{thm:spc-bound} but also well known, again in this case).
\end{xmpl}

\begin{xmpl}
  [mildly ill-posed operator]\label{xmpl:mild-type}
  Here we assume that the operator~$H$ has
logarithmic decay of the singular numbers, that is,\ $s_{j}(H)\asymp
\log^{-2p}j,\ p>0,\ j=1,2,\dots$, such that the operator is 'almost continuously invertible'. The index function for~$f_{0}$ is then given
as~$\varphi(t) = e^{-\beta/t^{1/(2p)}},\ t>0$. The inverse of the resulting
function~$\Theta_{\varphi}$ is seen to behave
like~$\Theta_{\varphi}^{-1}(s) \sim \beta^{2p}
\log^{-2p}\lr{\frac{1/s}{\log^p(1/s)}}$ using Lemma \ref{lem:geometry}.

  We consider again $\alpha$-regular priors, so that we find that~$s_j(\lmg) = s_j s_j(\lmf) \asymp  j^{-1- 2\alpha} \log^{-2p} j$. In particular Assumption~\ref{ass:decay-rate} holds, and Condition~(\ref{eq:alpha-beta-new}) is valid if and only if~$\alpha \leq \beta$. Thus, in the latter case the regularization bias is dominated, and the truncation 
  level~$\kn$ is obtained from balancing
$$
\frac k n \asymp \psi^{2}\lr{s_{k}(\lmg)} =
\Theta_{\varphi}^{2}\lr{s_{k}}\asymp k^{-2\beta}\log^{-2p}k,
$$
resulting
in~$k_n\asymp n^\frac{1}{1+2\beta}\log^{-\frac{2p}{1+2\beta}}(n)$, again using Lemma \ref{lem:geometry}.
Notice, that we do not need to explicitly determine the function~$\psi$ in this case, since the identity~$\psi^{2}\lr{s_{k}(\lmg)} =
\Theta_{\varphi}^{2}\lr{s_{k}}$ holds throughout, as mentioned above.
We
obtain that
\begin{equation}\label{eq:deltan-ex3}
\delta_n^2:=\spc^{\lpsi}(k_n)\lesssim \frac{\kn}{n} \asymp
n^{-\frac{2\beta}{1+2\beta}}\log^{-\frac{2p}{1+2\beta}}(n),
\end{equation}
and that this is the (square of the) minimax rate of statistical
estimation in the white noise model under smoothness expressed in
terms of the index function~$\Theta_\varphi$ from above. 
\end{xmpl}

Finally, we revisit Example~\ref{xmpl:severe-type}, but this time with
the covariance operator of the Gaussian prior as considered
in~\cite[Section 3.3]{MR3757524}.
\begin{xmplno}[Example~\ref{xmpl:severe-type} with analytic prior]
  The covariance operator of the Gaussian prior is assumed to have
  eigenvalues~$s_{j}(\lmf)\asymp j^{-\alpha} e^{-\xi j^{p}}$,
  $\xi>0, \alpha>0, p>0, j=1,\dots$. 
  
   Although the element~$g_{0}=Af_{0}$ is the same as before, i.e., $g_{0}$ has coefficients decaying at least as fast as $e^{-\gamma j^p}/j^{\beta}$, its smoothness relative to the resulting~$\lmg$ is with respect to a different function~$\psi$, such that again~$g_0\in \lpsi$.
  Indeed, we find that
  $s_j\lr{\lmg}\asymp j^{-\alpha}e^{-(\xi+2\gamma) j^p}$, so that
  again Assumption \ref{ass:decay-rate} only holds if $p\leq1$. We
  thus cannot apply Proposition \ref{prop:alphalessbeta} and we again
  need to explicitly check  which of the two terms dominates the
  bound \eqref{eq:spc-bound-thm} in Theorem~\ref{thm:spc-bound}. In particular, we again need to explicitly compute $\psi(t)=\Theta_\varphi((\Theta_\chi^2)^{-1}(t)).$

  By Assumption \ref{ass:link-noncomm}, we have that
  $\chi^2(t)\asymp t^\frac{\xi}{2\gamma}\log^{-\frac{\alpha}p}(1/t)$,
  so that
  $\Theta^2_{\chi}(t)\asymp t^{1+\frac{\xi}{2\gamma}}\log^{-\frac{\alpha}p}(1/t)$. Using
  Lemma \ref{lem:geometry}, we can invert $\Theta_{\chi}^2$ to get
  $$(\Theta^2_\chi)^{-1}(s)\sim
  s^{\frac{2\gamma}{2\gamma+\xi}}\log^{\frac{2\alpha\gamma}{p(2\gamma+\xi)}}(s^{-\frac{2\gamma}{2\gamma+\xi}}), \;\text{as}\ s\to0,$$
 thus
  \[\psi(t)\asymp
    t^\frac{\gamma}{2\gamma+\xi}\log^{-\frac{\beta}{p}+\frac{\alpha\gamma}{2\gamma
        p+\xi p}}(1/t), \;\text{as}\ t\to0.\] 
        
        On the one hand the regularization bias behaves asymptotically as
  \[\psi^2(1/n)\asymp
    n^{-\frac{2\gamma}{2\gamma+\xi}}\log^{-\frac{2\beta}{p}+\frac{2\alpha\gamma}{2\gamma
        p+\xi p}}(n).\]
  
 On the other hand, we find~$l_n$ from
  balancing
$$
\frac l n \asymp \psi^{2}\lr{s_{l}(\lmg)} =
\Theta_{\varphi}^{2}\lr{s_{l}}\asymp l^{-2\beta}e^{-2\gamma l^p},
$$
resulting in~$l_n\sim \lr{\frac{1}{2\gamma}\log(n)}^\frac 1 p$, again using Lemma
\ref{lem:geometry}.  As a result the second term in the bound
\eqref{eq:spc-bound-thm} is
$\frac{l_n}{n}\sim n^{-1}\lr{\frac{1}{2\gamma}\log(n)}^\frac 1 p$, which is always dominated by the regularization bias since $\xi>0$. Thus, the truncation~$\kn$ is obtained from~$s_k(\lmg)\asymp \frac 1 n$, resulting similarly in~$\kn \sim \lr{\frac{1}{\xi + 2\gamma}\log(n)}^\frac 1 p $ (which is smaller than $l_n$).

Combining, Theorem \ref{thm:spc-bound} gives
\begin{equation}\label{eq:deltan-ex3-analytic}
\delta_n^2:= \spc^{\lpsi}(k_n)\lesssim
n^{-\frac{2\gamma}{2\gamma+\xi}}\log^{-\frac{2\beta}{p}+\frac{2\alpha\gamma}{2\gamma
    p+\xi p}}(n).
\end{equation} 
In particular, this rate is worse than the (minimax) rate obtained by the $\alpha$-regular prior. 
\end{xmplno}

\subsection{Modulus of continuity and inverse rates}
\label{sec:examples}

  Below, we use Theorem~\ref{thm:phi-theta-bound} to bound the modulus at $f_0\in\eS \beta$, for $S=\Sk$ where $\Sk$ satisfies Assumption~\ref{ass:relations}, and for the three different choices of the linear operator~$H$. We then plug the rates~$\delta=\delta_n$ for the direct problem, obtained in the previous section, 
  into these bounds. According to Theorem \ref{thm:direct-inverse}, the resulting rates are rates of contraction for the corresponding inverse problem \eqref{eq:Ys} under the respective prior.
\begin{xmplno}[Example~\ref{xmpl:power-type} continued]
Here the setup is exactly the same as in
  Example~\ref{xmpl:power-type}, with $s_j:=s_j(H)\asymp j^{-2p}$, such that $\Theta_{\varphi}(t)=t^{(\beta + p)/(2p)}$.
For the (optimal) choice $k_\delta\asymp \delta^{-\frac{1}{\beta+p}}$, we thus get the bound on the modulus of continuity
\begin{equation}\label{eq:modex1}
\varphi\lr{\Theta_{\varphi}^{-1}(\delta)} \asymp
\delta^{\frac{\beta/(2p)}{\beta/(2p) +    1/2}} = \delta^{\beta/(\beta
  + p)},\quad \text{as}\ \delta\to 0.
\end{equation}  
Then, in order to get a rate of contraction for the original inverse problem with an $\alpha$-regular Gaussian prior truncated at $\kn$, it suffices to insert~$\delta_n$ from~\eqref{eq:deltan-ex1} into bound
\eqref{eq:modex1} on the modulus of continuity. Indeed, for $\alpha\leq \beta$ we get the rate
$$
\delta_n^{\frac{\beta}{\beta + p}}\lesssim n^{-\frac{\beta}{1 + 2\beta
    + 2p}},
$$
which is known to be the minimax rate in the inverse problem setting with the assumed moderately ill-posed operator $H$, under Sobolev-type smoothness $\beta$.
\end{xmplno}
\begin{xmplno}[Example~\ref{xmpl:severe-type} continued]
With the representation of~$\varphi$ and $\Theta_\varphi$ as in Example~\ref{xmpl:severe-type} we get the bound on the modulus of continuity
\begin{equation}\label{eq:modex2}
\varphi\lr{\Theta_{\varphi}^{-1}(\delta)} \asymp
\log^{-\beta/p}(1/\delta),\quad \text{as}\ \delta\to 0,
\end{equation}
which by again using Lemma~\ref{lem:geometry}, is achieved for $k_\delta\asymp \log^{1/p}(1/\delta).$

In order to get a rate of contraction for the original inverse problem with an $\alpha$-regular Gaussian prior truncated at ~$k_n$, 
it suffices to insert $\delta_n$ from~\eqref{eq:deltan-ex2} into the bound~\eqref{eq:modex2} on
the modulus of continuity. Regardless of whether~$\alpha\leq \beta$ or not, we get the rate
$$
\log^{-\beta/p}(1/\delta_n)\lesssim\log^{-\beta/p}(n),
$$
which is known to be the minimax rate in the inverse problem setting with the assumed severely ill-posed operator $H$, and under Sobolev-type smoothness $\beta$. That is, $\alpha$-regular Gaussian priors truncated at $\kn\asymp \log^\frac 1 p(n)$, are rate adaptive over Sobolev-type smoothness in this severely ill-posed operator setting. 

When using an analytic prior we need to insert the (sub-optimal) rate from~\eqref{eq:deltan-ex3-analytic} into bound \eqref{eq:modex2} on
the modulus of continuity. This yields that
$$
\log^{-\beta/p}(1/\delta_n)\lesssim\log^{-\beta/p}(n) ,
$$
is a rate of contraction for the inverse problem. In particular, the
truncated analytic Gaussian prior with truncation
point~$k_n\asymp \log^\frac 1 p(n),$ 
is also rate adaptive over Sobolev
balls $\eS\beta$, for all $\beta>0$. This is in agreement with the
findings in~\cite[Section 3.3]{MR3757524}.

\end{xmplno}
\begin{xmplno}[Example~\ref{xmpl:mild-type} continued]
With the representation of~$\varphi$ and $\Theta_\varphi$ as in Example~\ref{xmpl:mild-type}, and using Lemma \ref{lem:geometry} again, we observe an asymptotic behavior for the modulus of continuity
\begin{equation}\label{eq:modex3}
\varphi\lr{\Theta_{\varphi}^{-1}(\delta)} \asymp \delta
\log^p\lr{\frac 1 \delta},\quad \text{as}\ \delta\to 0,
\end{equation}
and this bound is achieved for $k_\delta\asymp \delta^{-\frac1\beta}\log^{-\frac{p}\beta}(1/\delta).$
This  is (up to a logarithm) linear in~$\delta$, and the inverse
problem is not much  harder than the direct one. In analogy to~\cite{MR3815105} the problem is~\emph{mildly ill-posed}. 

Inserting
the rate for~$\delta_n$ from~\eqref{eq:deltan-ex3} into bound \eqref{eq:modex3} on the modulus of continuity yields that
$$
\delta_n \log^{p}(1/\delta_n)\lesssim n^{- \frac{\beta}{1 + 2 \beta}}
\log^{{\frac{2\beta p}{1 + 2\beta}}}(n),
$$
is a rate of contraction for the inverse problem.
\end{xmplno}

\subsection{Discussion on the non-commuting case}
\label{sec:examples-discussion}
We conclude with a discussion about the non-commuting case, and we revisit the setup of Example~\ref{xmpl:power-type}, i.e.,\ with Sobolev type smoothness~$\beta$ and power type decay of the singular numbers of~$H$ as~$s_j(H) \asymp j^{-2p}$. In this case the applicability of Theorem~\ref{thm:direct-inverse} was limited to~$\mu \leq 2a + 1/2$, due to the assumed concavity of the function~$\psi$. Translating the assumed setup, we find that the exponent giving the smoothness of $f_{0}$ specifies to~$\mu:= \beta/(2p)$, while the exponent~$a$ in Assumption~\ref{ass:prior-linked2scale-noncomm} to~$a:= (1 + 2 \alpha)/(4p)$. First, the assumption~$a\geq1/2$ imposes a minimum regularity of the prior $1+2\alpha\geq 2p$, if $2p>1$. In terms of Sobolev smoothness~$\beta$, and for~$\alpha$-regular priors, the above limitation translates to~$\beta + p \leq 1 + 2(\alpha + p)$, and the function~$\psi$ would be given by~$\psi(t) = t^{(\beta + p)/(1 + 2 (\alpha + p))}$, being concave under this limitation. This is in accordance with the discussion at the end of~\S~\ref{sec:optimized-spc}, because when turning from~$f_0$ to~$g_0 = A f_0$  the Sobolev-type smoothness increases from~$\beta$ to~$\beta + p$. 
Also, the regularity of the prior, when turning from~$\lmf$ to~$\lmg$  increases from~$\alpha$ to~$\alpha + p$, see \eqref{eq:weyl-H-lmg}. 
Using this information to compute~$\kn$ from \eqref{eq:kn-def}, we get that the $\alpha$-regular prior truncated at $\kn\asymp n^\frac1{1+2\beta+2p}$ gives the minimax rate in this non-commuting setting, for~$\alpha\leq\beta\leq 1+2\alpha+p$.

\section{Proofs}
\label{sec:proofs}

In order to understand the arguments that are used in some of the subsequent proofs, we recall a few facts from the theory of
(bounded non-negative) self-adjoint operators in Hilbert space; we refer to~\cite{MR1477662} for a comprehensive treatment.
First, we introduce the partial
ordering for (non-negative) self-adjoint operators, say~$G^{1},G^{2}\colon \spZ \to
\spZ$, acting in a Hilbert space~$\spZ$. We write~$G^{1}\prec G^{2}$ if~$\scalar{G^{1}z}{z} \leq
\scalar{G^{2}z}{z},\ z\in \spZ$, and~$G^{1}\asymp G^{2}$ if
there are constants~$0 < a_{1},a_{2} < \infty$ such that both~$G^{1}
\prec a_{2}G^{2}$ and~$G^{2} \prec a_{1}G^{1}$. Weyl's Monotonicity Theorem, see~\cite[III.2.3]{MR1477662} asserts that~$G^1\prec G^2$ implies that the singular numbers also obey~$s_j(G^1) \leq s_j(G^2),\ j=1,2,\dots$
Furthermore, we
recall \emph{Heinz' Inequality}, see~\cite[Prop.~8.21]{MR1408680}, which states that for~$0 \leq \theta
\leq 1$ the inequality~$\norm{G^{1}z}{\spZ}\leq \norm{G^{2}z}{\spZ}$ implies
that~$\norm{\lr{G^{1}}^{\theta}z}{\spZ}\leq
\norm{\lr{G^{2}}^{\theta}z}{\spZ}$, where the fractional power is again
defined by spectral calculus.
We shall also use the fact that for a positive-definite, self-adjoint operator~$H\colon X \to~X$, an isometry~$U\colon X \to Y$, and an index function~$\zeta$, we have from spectral calculus  that
\begin{equation}\label{eq:speccalc}
\UU \zeta(H) \UU^\ast = \zeta(\UU H \UU^\ast). 
\end{equation}

Finally, the above ordering in the space of self-adjoint operators in Hilbert space gives rise to notions as \emph{operator monotonicity} and \emph{operator concavity}, extending the usual comparisons from real valued functions to self-adjoint operators by spectral calculus, and we refer to the monograph~\cite{MR1477662}. Specifically,  for some range, say~$[0,a]$, an operator valued function~$\psi$ is operator concave if for any pair of non-negative self-adjoint operators~$G^1, G^2$ with spectra in~$[0,a]$ it holds true that
$$\alpha \psi(G^1) + (1-\alpha) \psi(G^2)\prec \psi\lr{\alpha G^1 + (1-\alpha)G^2},\quad 0 \leq \alpha \leq 1.
$$
In our subsequent analysis we will confine to power type index functions. Such functions are operator concave if and only if they are concave. However, we occasionally use and highlight the relevance of the operator concavity to indicate that the results have extensions to the more general context, without dwelling into this.

\subsection{Proofs of Section~\ref{sec:direct}}

\begin{proof}[Proof of Proposition~\ref{prop:error-direct-unified}]
 The bound for the $\spc(g_{0},k)$ will be based on the decomposition
 in~(\ref{eq:spc-sum}), and we shall bound each summand, separately.
 
 We start with bounding the posterior spread, and notice that for a
 (non-negative finite rank) operator~$G\colon Y \to Y$ we always have that~$\tr{G} \leq
 (\myrank{G}) \norm{G}{Y\to Y}$.
Since the prior covariance~$\Ck$ has rank at most~$k$, and since~$\lr{\Ck + \frac 1 n}^{-1}\Ck$ is norm-bounded by one,  we can bound the posterior spread as
$$
\tr\lr{C_{post,k}} \leq \frac k n  \norm{\lr{\Ck + \frac 1 n}^{-1}\Ck}{Y \to
  Y} \leq \frac  k n.
$$

 Similarly we bound the estimation variance as 
  \begin{align*}
&\E\norm{\frac 1 {\sqrt n}\lr{\Ck + \frac 1 n}^{-1} \Ck\xi}{}^{2}
    = \frac 1 n \tr\left(\lr{\Ck +
                               \frac 1 n}^{-2} \Ck^{2} \right)\\
    &\leq  \norm{\lr{\Ck + \frac 1 n}^{-1} \Ck}{}\times\tr\lr{C_{post,k}}\leq \frac{k}{n}.
  \end{align*}

It remains to bound the estimation bias~$\norm{g_{0} - \E \hat g_{k}}{Y}$ under smoothness~$g_{0}\in~\lmg_{\psi}$. To this end we notice that~$\E(\hat g_k) =\lr{\Ck + \frac 1 n}^{-1} \Ck g_0$.
  Then the bias simplifies to
  \begin{equation}
    g_{0} - \E\hat g_{k} =\lr{I -\lr{\Ck +
                     \frac 1 n}^{-1} \Ck }g_{0}
                   =  \frac 1 n\lr{\Ck+
                     \frac 1 n}^{-1}g_{0}.\label{ali:bias}
  \end{equation}
 We introduce the residual function of Tikhonov
  regularization~$\ra(t) := \alpha/(t +~\alpha),$ $\alpha>0, t>0$, and it is readily checked that for a sub-linear index function~$\psi$ we have that~$\ra(t) \psi(t) \leq \psi(\alpha)$. This is then used by spectral calculus as operator function~$\ra(\Ck)$, which implies that~$\norm{\ra(\Ck)\psi(\Ck)}{Y\to Y}\leq \psi(\alpha)$.  Since $\norm{\ra(\Ck)}{Y \to Y}\leq~1$, this yields with~$\alpha:= 1/n$ and for~$g_0 \in\lpsi$, that
  \begin{align}
      \norm{g_{0} - \E\hat g_{k}}{Y} & = \norm{\ra(\Ck)g_0}{Y}\notag\\
      & \leq  \norm{\ra(\Ck)\psi(\Ck)}{Y\to Y} +\norm{\ra(\Ck)\lr{\psi(\lmg) - \psi(\Ck)}}{Y\to Y} \notag\\
      & \leq \psi\lr{\frac 1 n} + \norm{\psi(\lmg) - \psi(\Ck)}{Y\to Y},\label{ali:psi-diff}
  \end{align}
where the last inequality holds if the index function~$\psi$ is
sub-linear. Otherwise, the maximal decrease of the first summand (as~$n\to
    \infty$) is of the order~$\frac 1 { n}$, which is known as the saturation of Tikhonov regularization.

The second summand in~\eqref{ali:psi-diff} will be bounded, both for
the commuting (native prior or inherited prior 
with commuting $\lmf, \opT^\ast\opT$) and non-commuting (inherited
prior with non-commuting $\lmf, \opT^\ast\opT$) cases.  This will then
result in an overall bound for the~$\spc$ after taking into account
the bounds for the posterior spread and the estimation variance as
already established.

  First, in the native case, the projections~$Q_{k}$ are orthogonal on
  the singular subspaces of~$\lmg$, and we have that~$\psi(\Ck) = Q_{k}\psi(\lmg)$. Thus~$\psi(\lmg) - \psi(\Ck) = (I - Q_k) \psi(\lmg) $, and hence we have that
  $$
  \norm{\psi(\lmg) - \psi(\Ck)}{Y\to Y} = \psi\lr{s_{k+1}(\lmg)}.
  $$
  Thus, overall, from~\eqref{ali:psi-diff}, and the corresponding bounds for the posterior spread and the estimation variance, the $\spc^{\lpsi}(k,g_0)$ is bounded  by
     \begin{equation}
    \label{eq:spc-bound-2terms}
      \spc^{\lpsi}(k, g_{0}) \leq 
    \max\set{\psi^{2}\lr{\frac 1 { n}}, \frac \cb {n^2}} + 
    \psi^{2}\lr{s_{k+1}(\lmg)}   + 2\frac k n. 
    \end{equation}
  for some constant $\cb>0$, and this holds uniformly for~$g_0\in
  \lpsi$. The proof is complete, since~$1/n^2\leq k/n$.

  We turn to the case of inherited priors, and we shall use the
  operator concavity of the index function~$\psi$. This implies, cf.~\cite[Thm.~X.1.1]{MR1477662}, that the second summand in~\eqref{ali:psi-diff}, is bounded as
\begin{equation}\label{eq:psi-norm-diff}
\norm{\psi(\lmg) - \psi(\Ck)}{Y \to Y} \leq \psi\lr{\norm{\lmg - \Ck}{Y \to Y}}.
\end{equation}
We have that
$$
\lmg - \Ck = \opT \lmf \opT^{\ast} - \opT \lr{\lmf}^{1/2} P_{k} \lr{\lmf}^{1/2} \opT^{\ast} 
= \opT \lr{\lmf}^{1/2} \lr{I - P_k} \lr{\lmf}^{1/2} \opT^{\ast},
$$
which gives
$$
\norm{\lmg - \Ck}{Y \to Y} = \norm{\opT \lr{\lmf}^{1/2} \lr{I - P_k}}{X \to Y}^2.
$$
We thus bound  the approximation error
\begin{equation}\label{eq:rhok}
   \rho_k := \norm{\opT \lr{\lmf}^{1/2}(I - P_{k})}{X \to Y} ,
\end{equation}
which expresses the capability to approximate the compound operator~$\opT \lr{\lmf}^{1/2}$ by finite rank approximations, yielding by virtue of \eqref{eq:psi-norm-diff} that
\begin{equation}\label{eq:psi-norm-rhok}
\norm{\psi(\lmg) - \psi(\Ck)}{Y \to Y} \leq \psi\lr{\rho_k^{2}}.
\end{equation}
To this end, we will rely upon the link between~$\opT$ and~$\lmf$, as captured by Assumption~\ref{ass:prior-linked2scale-noncomm}.
  
  First, using Weyl's monotonicity Theorem with Assumption \ref{ass:prior-linked2scale-noncomm} we find that
  \begin{equation}
  \label{eq:weyl-H-lmf}
  s_{j}^{2a}(H) \asymp  s_{j}(\lmf),\quad j=1,2,\dots
\end{equation}
Next, by applying Heinz' Inequality with~$\theta= 1/(2a)\leq 1$, we see
that
\begin{equation}
  \label{eq:h12-lmf}
  \norm{H^{1/2}f}{X} \asymp \norm{\lr{\lmf}^{1/(4a)}f}{X},\quad x\in
  X.
\end{equation}

 Using spectral calculus and Assumption \ref{ass:prior-linked2scale-noncomm} with~$f:= \opT^{\ast}g,\ g\in Y$ we find for
arbitrary~$g\in Y$ that 
\begin{equation}
  \label{eq:aast-lmg}
\norm{\lr{\opT \opT^{\ast}}^{a+1/2} g}{Y} = \norm{ H^{a} \opT^{\ast}g}{{X}} \asymp
\norm{\lr{\lmf}^{1/2} \opT^{\ast}g}{{X}} = \norm{\lr{\lmg}^{1/2}g}{Y}.  
\end{equation}

Thus, Weyl's Monotonicity Theorem yields
\begin{equation}
  \label{eq:weyl-H-lmg}
s_{j}^{a+1/2}(H) =  s_{j}^{a+1/2}(\opT \opT^{\ast}) \asymp s_{j}^{1/2}(\lmg),\quad j=1,2,\dots
\end{equation}
We shall use these estimates, and the fact that~$P_{k}$ are the singular projections of~$\lmf$,  to bound~$\rho_{k}$ as
\begin{align*}
  \rho_{k} & = \norm{\opT \lr{\lmf}^{1/2}(I - P_{k})}{X \to Y} =
             \norm{H^{1/2}\lr{\lmf}^{1/2}(I - P_{k})}{X \to X}\\
  & \asymp \norm{\lr{\lmf}^{1/(4a)}\lr{\lmf}^{1/2}(I - P_{k})}{X \to
    X} = s_{k+1}^{\frac{2a+1}{4a}}(\lmf)\\
  & = s_{k+1}^{a+1/2}(H) \asymp s_{k+1}^{1/2}(\lmg),\quad k=1,2,\dots
\end{align*}
Thus~$\rho_k^{2}\asymp s_{k+1}(\lmg)$, as~$k\to\infty$. Inserting this
into the bound from~(\ref{eq:psi-norm-rhok}) we complete the
estimate for the bias from~\eqref{ali:psi-diff}, and obtain the same
bound as in the native case, when restricting to operator
concave~$\psi$. This completes the proof.
\end{proof}
\begin{rem}
Within the context of projection schemes for ill-posed equations in Hilbert space, a more elaborate analysis allows for bounding the bias for general spectral regularization schemes, and for certain index functions which can express higher order smoothness. Specifically, such index functions are products of operator concave and Lipschitz ones; we refer to~\cite[Thm.~2]{MR2036530} for details.
\end{rem}

\begin{proof}[Proof of Theorem~\ref{thm:spc-bound}]
Let~$\kn$ be as in~\eqref{eq:kn-def} and consider the right hand side of~\eqref{eq:spc-bound}. We then observe that for the index~$\kn +1$ we have that~$\psi^2\lr{s_{\kn+1}(\lmg)} \leq \max\set{\psi^{2}\lr{\frac 1 n},\frac {\kn+1} n}$. 
   For the proof we shall distinguish two cases. First, we shall assume that~$\psi^2(1/n) \leq \kn/n$. In this case we bound
   \begin{align*}
       \psi^2\lr{\frac 1 n} + \psi^2\lr{s_{\kn+1}(\lmg)} + \frac{\kn}{n} 
     & \leq 2 \frac{\kn}{n} + \psi^{2}\lr{s_{\kn+1}(\lmg)}\\
     & \leq 2 \frac{\kn}{n} + \frac{\kn +1}{n}\leq 4 \frac{\kn}{n}.
   \end{align*}
   In the other case, when~$1/n \leq \kn/n < \psi^2(1/n)$, we bound
   \begin{align*}
       \psi^2\lr{\frac 1 n} + \psi^2\lr{s_{\kn+1}(\lmg)} + \frac{\kn}{n} 
     & \leq 2 \psi^2\lr{\frac 1 n} + \max\set{\psi^2\lr{\frac1n}, \frac{\kn}n+\frac1n}\\
     & \leq 4 \psi^2\lr{\frac 1 n}.
   \end{align*}
   Thus, in either case we find that
   $$
   \psi^2\lr{\frac 1 n} + \psi^2\lr{s_{\kn+1}(\lmg)} + \frac{\kn}{n} \leq 4 \max\set{\psi^{2}\lr{\frac 1 n},\frac {\kn} n},
   $$
  and by Proposition~\ref{prop:error-direct-unified} we get the bound~\eqref{eq:spc-bound-thm} in both settings considered in the statement.
   In order to assert that the
  contraction rate is order optimal in the case when ~$\psi^2(1/n) \leq \kn/n$, we use the fact that
  $$
\inf_{k} \lr{\psi^{2}\lr{s_{k+1}(\lmg)} +k/n} \geq \frac {\kn}{n},
$$
together with Proposition~\ref{prop:donohoetal}. The last bound is seen as
follows. First, if~$k\geq \kn$ the above bound is trivial. If, on the
other hand~$k < \kn$, yielding that~$k+1 \leq \kn$,  then
$$
\psi^{2}\lr{s_{k+1}(\lmg)} +k/n \geq \psi^{2}\lr{s_{\kn}(\lmg)} \geq \frac{\kn}{n},
$$
by the definition of~$\kn$ in~(\ref{eq:kn-def}).
\end{proof}


\begin{proof}[Proof of Proposition~\ref{prop:alphalessbeta}]
Notice that at~$\kn+1$ we have that~\begin{equation}\label{eq:kn+1}\psi^{2}\lr{s_{\kn+1}(\lmg)} \leq \max\set{\psi^2\lr{\frac 1 n },\frac {\kn +1} n}. \end{equation}

We  first assume~\eqref{eq:alphabeta} to hold and show that $\kn/n$ dominates in \eqref{eq:spc-bound-thm}.
If~$\psi^2(1/n) \leq (\kn+1)/n$, then we find
\begin{equation}\label{eq:lower-order}
\psi^2(1/n) \leq \frac{\kn+1}{n}\leq \cc\cd \frac{\kn}{n}
\end{equation}
for~$n \geq n_0:=1/(\cc\cd -1)$, 
which means that $\kn/n$ dominates in \eqref{eq:spc-bound-thm}.

  Otherwise, if~$(\kn+1)/n \leq \psi^2(1/n)$ then \eqref{eq:kn+1} gives that~$\psi^{2}\lr{s_{\kn+1}(\lmg)} \leq \psi^2(1/n)$, thus~$s_{\kn+1}(\lmg) \leq 1/n$. Using~\eqref{eq:kn-def},~\eqref{eq:alphabeta} and~\eqref{eq:decay-rate}, we bound
$$
\psi^2\lr{\frac 1 n} < \psi^2\lr{s_{\kn}(\lmg)}  \leq \cd \kn s_{\kn}(\lmg) \leq \cc\cd \kn s_{\kn+1}(\lmg) \leq \cc\cd \frac{\kn}{n},
$$
which again proves that $\kn/n$ dominates in \eqref{eq:spc-bound-thm}.

For the converse implication, suppose that~\eqref{eq:alphabeta} is violated. 
Again, if $\frac{\kn~+~1}n~\leq~\psi^2(1/n)$ then
$$
\frac{\kn}{n} \leq \frac{\kn+1}{n}\leq \psi^2(1/n),
$$
showing that the regularization bias $\psi^2(1/n)$ dominates in \eqref{eq:spc-bound-thm}.

Otherwise, if~$\psi^2(1/n) \leq (\kn+1)/n$, then by \eqref{eq:kn+1} we find that
$\psi^2\lr{s_{\kn +1}(\lmg)} \leq \frac{\kn+1} n$, 
 or equivalently that
$$
n \leq \frac {\kn+1} {\psi^{2}\lr{s_{\kn + 1}(\lmg)}}.
$$
The violation of~\eqref{eq:alphabeta} yields that
\begin{equation}\label{eq:quottozero}
\frac{j s_{j}(\lmg)}{\psi^{2}\lr{s_{j}(\lmg)}} \to 0\quad \text{as}\ j\to\infty.
\end{equation}
Hence, first using Assumption~\ref{ass:decay-rate}, we can bound
\begin{align*}
    n s_{\kn}(\lmg)& \leq \cc n s_{\kn + 1}(\lmg)\leq \cc \frac{(\kn +1) s_{\kn + 1}(\lmg)}{\psi^{2}\lr{s_{\kn +1}(\lmg)}} \longrightarrow 0,
\end{align*}
 by virtue of~\eqref{eq:quottozero}, because~$\kn \to \infty$, see Remark~\ref{rem:kn-infinite}.
Thus,  we must have that $s_{\kn}(\lmg) \leq \frac 1 n$, for $n$ sufficiently large. Overall, we conclude that then
$$
\frac {\kn}{n} < \psi^2\lr{s_{\kn}(\lmg)}
\leq \psi^{2}\lr{\frac 1 n},
$$
showing that the regularization bias dominates~$\kn/n$, as $n\to\infty$ and thus~$\kn\to~\infty$.
\end{proof}

\subsection{Proofs of Section~\ref{sec:inverseP}}

In order to establish the bound for the modulus of continuity, we rely
on the following auxiliary result, bounding the modulus of continuity in terms of
the degree of approximation and the modulus of injectivity, as
introduced in Section~\ref{sec:main-bound}.
\begin{prop}\label{prop:main-error-bound}
  Let~$f_{0}\in X$. The following bound holds true
for every~$k\in\nat$ and~$h\in \Sk$.
\begin{equation}
  \label{eq:main-error-Xn}
\norm{h - f_{0}}{X} \leq \lr{1 + 
  \frac{\varrho{(H^{1/2},\Sk)}}{j(H^{1/2},\Sk)}}\norm{(I -
  P_{k})f_{0}}{X} + 
\frac{\norm{H^{1/2}(h - f_{0})}{X}}{j(H^{1/2},\Sk)}.  
\end{equation}
\end{prop}
\begin{proof}  Given~$f_{0}$ we assign~$f_{k}:= P_{k}f_{0}$. Clearly, both~$h, f_{k}\in \Sk$. Then  we can bound
$$
\norm{h - f_{0}}{} \leq \norm{h - f_{k}}{} + \norm{f_{k} - f_{0}}{}
= \norm{h - f_{k}}{} + \norm{(I - P_{k})f_{0}}{}.
$$
For the first summand we continue and bound, noticing that~$P_{k}(h -
f_{k}) = h - f_{k}$ as
\begin{align*}
 \norm{h - f_{k}}{} & = \norm{P_{k}(h-  f_{k})}{} \leq
                       \frac{\norm{H^{1/2}(h - f_k)}{}}{j(H^{1/2},\Sk)}\\ 
& \leq  \frac{\norm{H^{1/2}(h - f_{0})}{} + \norm{H^{1/2}(f_{0} -f_{k})}{} }{j(H^{1/2},\Sk)}\\ 
& \leq  \frac{\norm{H^{1/2}(h - f_{0})}{}}{j(H^{1/2},\Sk)} + \frac{\norm{H^{1/2}(f_{0} -
  f_{k})}{}}{j(H^{1/2},\Sk)}\\
& \leq  \frac{\norm{H^{1/2}(h - f_{0})}{}}{j(H^{1/2},\Sk)}+ 
  \frac{\varrho(H^{1/2},\Sk)\norm{(I -
  P_{k})f_{0}}{}}{j(H^{1/2},\Sk)}.
\end{align*}
This completes the proof.
\end{proof}
\begin{proof}[Proof of Theorem~\ref{thm:phi-theta-bound}]
  Given the bound from Proposition~\ref{prop:main-bound} there is a
  constant~$\cb$ such that
  $$
\omega_{f_{0}}(H^{-1/2},\Skd,\delta) \leq  \cb 
\lr{{\varphi}(s_{\nast+1}) + \frac{\delta}{\sqrt{s_{\nast}}}}.
$$
(We can take~$\cb:= \max\set{M\lr{1 + C_{P}C_{B}
  },C_{B}
}$.)

 Let~$\tast$ be the solution to the equation~$\Theta_{\varphi}(t) = \delta$,
 and $\nast$ be given as in~(\ref{eq:nast}). Notice that~$\tast \to 0$ as~$\delta\to 0$, and hence that~$\nast\to\infty$ as~$\delta\to 0$.

 First we see that~$\Theta_{\varphi}\lr{s_{\nast +1} }\leq \delta$, and hence
 $\varphi(s_{\nast+1})\leq \varphi\lr{\Theta_{\varphi}^{-1}(\delta)}$.
 Also, $s_{\nast} >\tast$, and 
 therefore we find that
 $$
 \frac{\delta}{\sqrt{s_{\nast}}} < \frac{\delta}{\sqrt{\tast}}
 =
 \frac{\Theta_{\varphi}\lr{\Theta_{\varphi}^{-1}(\delta)}}{\sqrt{\Theta_{\varphi}^{-1}(\delta)}}
 = \varphi\lr{\Theta_{\varphi}^{-1}(\delta)}.
 $$
 Overall this results in the desired bound  with~$\cg:= 2 \cb$.
\end{proof}

\subsection{Proofs of Section~\ref{sec:relating}}

\begin{proof}  [Proof of Lemma~\ref{lem:commuting-H-lmf}]
This is a consequence of the
simultaneous diagonalization Theorem, see e.g.~\cite{MR2978290}. If~$H=\asta$ and~$\lmf$ commute then this also holds true for
the projections~$P_{k}$, because these are singular with respect
to~$\lmf$. The polar decomposition of~$\opA$ yields an isometry~$\UU:X\to Y$ such that~$\opA = \UU H^{1/2}$, so that 
$$
\lmg:=\opA\lmf\opA^\ast=\UU H^{1/2}\lmf H^{1/2}\UU^\ast, \;\lmg:Y\to Y.
$$
We thus see that
$$
\Ck= \opA P_{k} \lmf P_{k}\opA^{\ast} =  \UU P_{k} \UU^{\ast}
 \UU H^{1/2} \lmf H^{1/2}\UU^{\ast} \UU P_{k} \UU^{\ast}=\UU P_{k}\UU^\ast\lmg \UU P_{k}\UU^\ast,
 $$
where~$Q_{k}:= \UU P_{k} \UU^{\ast}$ are orthogonal projections onto the spaces~$U \Xk\subset~Y$, which coincide with the singular spaces of $\lmg$.
Hence
 $$
\Ck= Q_{k}\lmg Q_{k},
$$
such that the push-forward prior~$\opA_{\sharp \lr{\Pf}}$ is native
for~$g$.

We next express the covariance $\lmg$ via $H$. To this end, we use Assumption \ref{ass:link-noncomm}, quantifying the commutativity between~$H$ and~$\lmf$ by
  assuming that these are linked via a certain index
  function~$\chi$. Then, recalling from Definition
  \ref{de:index-noncomm} that  $\Theta_{\chi}$ denotes the companion
  $\chi$, we can write, using~\eqref{eq:speccalc}, that
\begin{equation}
  \label{eq:lmg-proof}
 \lmg= \UU H^{1/2} \lmf H^{1/2} \UU^{\ast} = U \Theta_{\chi}^{2}(H)
   \UU^{\ast}= \Theta_{\chi}^{2}(\UU H\UU^\ast).  
\end{equation}
 For the last equality we used.
In particular, identity~\eqref{eq:lmg-proof} yields $s_j(\lmg) = \Theta_{\chi}^{2}(s_j),\ j=1,2,\dots$, where we recall that $s_j:=s_j(H).$

Assumption~\ref{ass:link-noncomm} also allows us to translate a given source
condition (smoothness class) for $f_0\in X$ relative to the
operator~$H$, to a source condition for $g_0=\opA f_0 = \UU H^{1/2}f_0\in Y$,
relative to the operator~$\lmg$. 
More precisely, we shall identify an
index function~$\psi$ such that $g_0\in\lmg_\psi.$
Indeed, for any element~$f_{0}\in {\af}$ we find, with~$w:= U v\in Y$, and using~\eqref{eq:lmg}, that
\begin{equation}
  \label{eq:g0inlambda}
g_{0}:= \UU H^{1/2}f_{0} = \UU \Theta_{\varphi}(H)\UU^\ast (U v) =
\Theta_{\varphi}\lr{\lr{\Theta_{\chi}^{2}}^{-1}(\lmg)}w,\quad \norm{w}{Y}\leq 1,
\end{equation}
where $\Theta_{\varphi}$ is the companion of $\varphi$, see Definition \ref{de:index-noncomm}. This corresponds to~$g_{0}\in \lmg_{\psi}$ with index function~$\psi(t):=
\Theta_{\varphi}(\lr{\Theta_{\chi}^{2}}^{-1}(t)),\ t>0$.
\end{proof}
\begin{proof}[Proof of Lemma~\ref{lem:smoothness-h2lmg}]
  Here we start from~(\ref{eq:aast-lmg}) (with~$\opT := \opA$), and apply Heinz' Inequality
  with~$\theta:= \frac{\mu+1/2}{a + 1/2}\leq 1$. This gives
  $$
  \norm{ H^{\mu} \opA^{\ast}g}{X} =\norm{\lr{\opA \opA^{\ast}}^{\mu+1/2} g}{Y}
 \asymp \norm{\lr{\lmg}^{\frac{\mu+1/2}{2a+1}}g}{Y},\quad g\in Y.   
 $$
 By virtue of Douglas' Range Inclusion Theorem, and we refer
 to~\cite{MR3985479}, this implies
 $$
\mathcal R\lr{\opA H^{\mu}} = \mathcal R\lr{\lr{\lmg}^{\frac{\mu+1/2}{2a+1}}},
$$
such that a source-wise representation $f_0=H^\mu v, \;v\in X$, yields a corresponding representation $g_{0}=\opA f_{0} =(\lmg)^{\frac{\mu+1/2}{2a+1}} w, \;w\in Y$.
Finally, under~$\mu\leq a$ we see that~$\frac{\mu +
      1/2}{2a+1}\leq 1/2$ which yields the operator concavity of~$\psi^{2}$.
  \end{proof}

\begin{proof}[Proof of Lemma~\ref{lem:smoothness-h2lmg-3over2}]
  First, as argued before, the relation
  in~\eqref{eq:3over2} implies the validity of
  Assumption~\ref{ass:prior-linked2scale-noncomm}, such
  that~(\ref{eq:aast-lmg}) holds true, and we have that
  \begin{equation}
    \label{eq:lmg12}
    \norm{H^{a}\opA^{\ast}g}{X} \asymp
    \norm{\lr{\lmg}^{1/2}g}{Y},\quad g\in Y.
  \end{equation}
  Applying~\eqref{eq:h12-lmf} (which holds under Assumption~\ref{ass:prior-linked2scale-noncomm}) to~$f:=
  \lmf \opA^{\ast}g,\ g\in Y$ we infer that
  \begin{equation}
    \label{eq:lmg-lmf}
    \norm{\lmg g}{Y} \asymp \norm{\lr{\lmf}^{(4a +1)/(4a)}\opA^{\ast}g}{X},\quad g\in Y.
  \end{equation}
  Now we actually use~(\ref{eq:3over2}). Indeed, for~$a\geq 1/2$ we
  find that~$(4a +1)/(4a) \leq 3/2$. Thus, we can apply Heinz'
  Inequality with~$\theta:= 2(4a +1)/(12a)\in[0,1]$ to~(\ref{eq:3over2}) which
  gives
  \begin{equation}
    \label{eq:lmg1}
 \norm{\lmg g}{Y} \asymp \norm{\lr{\lmf}^{(4a +1)/(4a)}\opA^{\ast}g}{X}
\asymp \norm{ H^{2a+1/2} \opA^{\ast}g}{X},\quad g\in Y.   
  \end{equation}
Thus we have two inequalities (derived from~(\ref{eq:lmg12})
and~(\ref{eq:lmg1}), respectively) and with
generic constant~$C$), namely
\begin{align*}
  \norm{H^{a} \opA^{\ast}g}{X} & \leq C \norm{\lr{\lmg}^{1/2}g}{Y},\quad g\in Y,
                            \intertext{and  also}        
  \norm{H^{2a+1/2}\opA^{\ast}g}{X} & \leq C \norm{\lmg g}{Y}, \quad g\in Y.
\end{align*}
We are thus in the setting of~\cite[Thm.~3]{MR2277542} of
interpolation in Hilbert scales, and we conclude
that for~$a \leq \mu \leq 2a +1/2$ we have that
\begin{equation}
  \label{eq:interpol-bound32}
   \norm{H^{\mu} \opA^{\ast}g}{X}  \leq C
   \norm{\lr{\lmg}^{(\mu + 1/2)/(2a +1)} g}{Y}, \quad g\in Y.
 \end{equation}
 Again, Douglas' Range Inclusion Theorem asserts that then
 $$
 \mathcal
 R\lr{\opA H^{\mu}} \subseteq \mathcal
 R\lr{\lr{\lmg}^{(\mu + 1/2)/(2a +1)}}.
 $$
 In other words, every element~$g_{0}=\opA H^{\mu} v$
 belongs to (a multiple of)~$\lpsi$ for the function~$\psi(t)=
 t^{(\mu + 1/2)/(2a +1)}$, and the proof is complete.
\end{proof}
  \begin{proof}[Proof of Proposition~\ref{prop:validity-ass-relations}]

    If Assumption~\ref{ass:link-noncomm} holds, then the operator~$\lmf$ and~$H$ commute, and hence the spaces~$\Xk$ are also singular spaces for~$H$. In this case we have that
    $$
    j(H^{1/2},\Xk) = s_k,\quad \varrho(H^{1/2},\Xk) = s_{k+1},\quad \text{and}\ \norm{(I - P_k) \varphi(H)}{X} = \varphi(s_{k+1}),
    $$
    such that Assumption~\ref{ass:relations} holds with~$C_P = C_B = M = 1$.

 Under Assumption~\ref{ass:prior-linked2scale-noncomm} we
  use~(\ref{eq:h12-lmf}) with~$f:= (I - P_{k})v,\ \norm{v}{X}\leq 1$
$$
\norm{H^{1/2}(I - P_{k})v}{X} \asymp \norm{\lr{\lmf}^{1/(4a)}(I -
  P_{k})v}{X}\leq s_{k+1}^{1/(4a)}(\lmf) \asymp s_{k+1}^{1/2},
$$
where we used~(\ref{eq:weyl-H-lmf}) for the last asymptotics. This
shows the Jackson Inequality. Similarly, using~(\ref{eq:h12-lmf})
with~$f\in \Xk,\ \norm{f}{X}=1$ we bound
$$
\norm{H^{1/2}f}{X} \asymp \norm{\lr{\lmf}^{1/(4a)}f}{X} \geq
j\lr{\lr{\lmf}^{1/(4a)},\Xk} = s_{k}^{1/(4a)}(\lmf) \asymp
s_{k}^{1/2},
$$
which shows that a Bernstein Inequality also holds true. Finally, for
a power type function~$\varphi(t) = t^{\mu}$ and for~$0 < \mu \leq a$,
Heinz' Inequality with~$\theta:= \mu/a$ applied to the asymptotics in
Assumption~\ref{ass:prior-linked2scale-noncomm} yields
$$
\norm{H^{\mu}(I - P_{k})v}{X} \asymp \norm{\lr{\lmf}^{\mu/(2a)}(I -
  P_{k})v}{X}\leq R s_{k+1}^{\mu/(2a)}(\lmf) \asymp s_{k+1}^{\mu},
$$
whenever~$\norm{v}{X}\leq R$.
Similarly, under the stronger assumption~\eqref{eq:3over2} and for the extended range $0<\mu\leq 2a+1/2$, we find by using Heinz' Inequality with~$\theta=\mu/(3a)$ that
$$
\norm{H^{\mu}(I - P_{k})v}{X} \asymp \norm{\lr{\lmf}^{\mu/(2a)}(I -
  P_{k})v}{X}\leq R s_{k+1}^{\mu/(2a)}(\lmf) \asymp s_{k+1}^{\mu}.
$$ 
Consequently, whenever~$f_0\in\af$ for~$\varphi(t)= t^{\mu}$ with~$0 < \mu \leq 2a+1/2$ (under the stronger assumption \eqref{eq:3over2} when $\mu >a$), it holds
$$
\norm{(I - P_{k})f_{0}}{X} \leq R \norm{(I - P_{k})H^{\mu}}{X\to X}
=  R \norm{H^{\mu}(I - P_{k})}{X\to X}  \lesssim  s_{k+1}^{\mu},
$$
which completes the proof.  
\end{proof}

\begin{proof}[Proof of Proposition~\ref{prop:relation}]
   By Proposition \ref{prop:validity-ass-relations}, the spaces~$(\Xk)_{k\in\mathbb N}$ satisfy
   Assumption~\ref{ass:relations}
  (with~$C_{B}=C_{P}=M=1$), and hence Proposition~\ref{prop:main-bound}
  applies and yields
  \begin{align*}
    \omega_{f_0}\lr{H^{-1/2}, \Xkn,\delta_n} & \leq 2\lr{\varphi(s_{\kn+1}) + \frac{\delta_n}{\sqrt {s_{\kn}}}}.
  \end{align*}
 We bound the two summands. By the definition of~$\kn$, and recalling from §~\ref{sec:commute} that $s_{j}(\lmg)=\Theta^2_{\chi}(s_j),$ we find
    that
    $$
    \psi^2(\Theta_{\chi}^2(s_{\kn+1})) \leq \max\set{\psi^2(1/n), \frac{\kn+1}{n}}\leq 2 \max\set{\psi^2(1/n), \frac{\kn}{n}}\leq \delta_n^2.
    $$
    This yields~$\Theta_{\varphi}(s_{\kn +1}) \leq \delta_n$, and hence that
  $$
  \varphi(s_{\kn+1}) \leq
  \varphi\lr{\Theta_{\varphi}^{-1}(\delta_n)}.
  $$
  To bound the second summand  we recall the definition of $\kn$ to see that
  $$
  \Theta_{\varphi}^{2}(s_{\kn}) = \psi^2\lr{s_{\kn}(\lmg)} > \max\set{\psi^2(1/n),\frac{\kn}{n}}\geq \frac{\delta_n^2}{\ci^2}.
  $$
 Thus we see that~$s_{\kn} \geq \lr{\Theta_{\varphi}^{2}}^{-1}\lr{\frac{\delta_n^2}{\ci^2}}$, and consequently
  \begin{align*}
       \frac{\delta_n^{2}}{s_{\kn}} & \leq \frac{\delta_n^{2}}{\lr{\Theta_{\varphi}^{2}}^{-1}\lr{\frac{\delta_n^2}{\ci^2}}}
       = \ci^2 \frac{\Theta_\varphi^2\lr{\Theta_\varphi^2}^{-1}\lr{\frac{\delta_n^{2}}{\ci^2}}}{\lr{\Theta_{\varphi}^{2}}^{-1}\lr{\frac{\delta_n^2}{\ci^2}}}\\
       & = \ci^2 \varphi^2\lr{\Theta_\varphi^2}^{-1}\lr{\frac{\delta_n^{2}}{\ci^2}}
       \leq \ci^2 \varphi^2\lr{\Theta_\varphi^2}^{-1}\lr{\delta_n^{2}},
  \end{align*} 
  where for the last bound we used that both $\varphi$ and $\Theta_\varphi$ are non-decreasing.
  The proof can thus be completed.
\end{proof}

The proof of Proposition~\ref{prop:relation} above, consisted of three steps.
First, we used Assumption~\ref{ass:relations} in order to derive a
bound for the modulus of continuity in terms of a decreasing (in~$k$)
smoothness-dependent part, and a non-decreasing part. Then, each of the two terms
were appropriately bounded by using the definition of~$\kn$. We follow
a similar strategy in the next proof as well.
\begin{proof}  [Proof of Proposition~\ref{prop:relation-noncomm}]
  By virtue of Propositions~\ref{prop:validity-ass-relations} and~\ref{prop:main-bound}, we have the
  following error bound for the modulus of continuity:
  $$
  \omega_{f_0}\lr{H^{-1/2}, \Xkn,\delta_n}  \lesssim
  {\varphi(s_{\kn+1}) + \frac{\delta_n}{\sqrt {s_{\kn}}}}.
  $$
  By the definition of~$\kn$ from~(\ref{eq:kn-def}) we shall derive an
  upper bound for~$s_{\kn+1}$, and a lower bound for~$s_{\kn}$. For
  this we recall from~(\ref{eq:weyl-H-lmg}) that~$s_{j}^{2a+1} \asymp
  s_{j}(\lmg),\ j=1,2,\dots$. So, by the choice of~$\kn$ and
  from~(\ref{eq:mild-ass}) we see
$$
s_{\kn +1}^{2\mu + 1} \lesssim \psi^{2}\lr{s_{\kn +1}(\lmg)} \leq
\frac 1 2 \delta_{n}^{2},
$$
such that we find~$s_{\kn+1}^{\mu} \lesssim \delta_{n}^{\mu/(\mu +
  1/2)}$, bounding the decay of the smoothness-dependent term.

It remains to lower bound~$s_{\kn}$. Again from~(\ref{eq:kn-def}) and~(\ref{eq:mild-ass}) we
see that
$$
\psi^{2}\lr{s_{\kn}(\lmg)} > \frac{1}{\ci^2}\delta_{n}^{2},
$$
which yields~$s_{\kn} \gtrsim \delta_{n}^{\frac 1 {\mu + 1/2}}$. This
in turn yields~$\delta_{n}/s_{\kn} \lesssim \delta_{n}^{\mu/(\mu +
  1/2)}$, hence completing the proof.
\end{proof}
\appendix

\section{Extending the bounds for the modulus of continuity}\label{app:A}

The main bound as established in Theorem~\ref{thm:phi-theta-bound}
allows for the following variation, and yields an analog and
extension of the bound~(3.4) in~\cite{MR3757524} for general
smoothness assumptions and for 
arbitrary decay rates of the singular numbers.
 To this end we enrich the subspaces~$\Sk$ as follows.
For a (decreasing positive) sequence~$\rho =
\lr{\rho_{k}}_{k\in\nat}$ and for~$k\in\nat$ we assign
$$
\Sn k :=
\set{f\in X,\quad \norm{(I-P_{k})f}{X}\leq \rho_{k}},
$$
where again we let~$P_{k}$ be the orthogonal projection onto the
subspace~$\Sk$.
\begin{prop}
  \label{prop:KS-bound}
  Suppose that~$f_{0}\in\af$, and that Assumption~\ref{ass:relations}  holds. Given~$\delta>0$ let~$\nast$ be as in~(\ref{eq:nast}).
 If the sequence~$\rho$ obeys~$\rho_{\nast} =
\bigo\lr{\varphi\lr{\Theta_{\varphi}^{-1}(\delta)}}$ as~$\delta\to
0$
then there is a constant~$\ch$ such that 
  $$
\omega_{f_{0}}(H^{-1/2},\Sn {\nast},\delta) \leq \ch \varphi\lr{\Theta_{\varphi}^{-1}(\delta)}.
$$ 
\end{prop}
\begin{proof}
  Indeed, for~$f\in \Sn k$, we
let~$h:= P_{k}f\in \Sk$. Then we find that
$$
\norm{f - f_{0}}{} \leq \norm{f - h}{} + \norm{h - f_{0}}{}\leq
\rho_{k} +  \norm{h - f_{0}}{},
$$
such that as~$\delta\to 0$ we have that
$$
 \omega(H^{-1/2},\Sn {\nast},\delta) \leq \rho_{\nast} +
 \omega(H^{-1/2},\Skd,\delta)
 = \bigo\lr{\varphi\lr{\Theta_{\varphi}^{-1}(\delta)}},
 $$
 where we used the bound from Theorem~\ref{thm:phi-theta-bound}.
\end{proof}

We also mention the following extension to non-linear mappings.
\begin{ass}[non-linearity structure]\label{ass-non-linear}
There are
constants~$\ce\leq \cff<\infty$ such that
\begin{equation}
  \label{eq:non-linear-K}
  \ce  \norm{H^{1/2}(f - f_{0})}{X} \leq \norm{A(f) -
    A(f_{0})}{Y} \leq \cff  \norm{H^{1/2}(f - f_{0})}{X} ,\quad f\in \mathcal D(A),
\end{equation}
holds true, and that the element~$f_{0}\in \mathcal D(A)$ is an
interior point in the domain of~$A$, i.e.\ there is~$R>0$ such that~$B_{R}(f_{0})\subset \mathcal
D(A)$. 
\end{ass}

Such non-linearity assumption was made in a different context
in~\cite{MR3742371}. 
We will not pursue this path within the present study, but this may
lead to an extension of the present analysis to some class of
non-linear problems. We mention the following bound, for which we recall the notion of \emph{degree of approximation}~$\varrho$, as well as the \emph{modulus of injectivity~$j$} from Section~\ref{sec:main-bound} (here tentatively extended to non-linear mappings). This is an extension of Proposition~\ref{prop:main-error-bound} to non-linear mappings~$\opA$.

\begin{prop}
  Let~$f_{0}\in X$.
Under Assumption~\ref{ass-non-linear} the following bound holds true
for every~$k\in\nat$ and~$h\in \X_{k}\cap \mathcal D(\opA)$.
\begin{equation}
  \label{eq:main-error-Xn-non-linear}
\norm{h - f_{0}}{X} \leq \lr{1 + \frac \cff \ce
  \frac{\varrho(H^{1/2},\X_{k})}{j(H^{1/2},\X_{k})}}\norm{(I -
  P_{k})f_{0}}{X} + 
\frac 1 \ce \frac{\norm{\opA(h) - \opA{f_{0}})}{Y}}{j(H^{1/2},\X_{k})}.  
\end{equation}
\end{prop}

\begin{proof}  Given~$f_{0}$ we assign~$f_{k}:= P_{k}f_{0}$. Clearly, both~$h, f_{k}\in \X_{k}$, and for~$k$ large enough we will have~$\norm{f_0 - f_k}{}\leq R$, such that~$f_k\in\mathcal D(\opA)$. Then  we can bound
$$
\norm{h - f_{0}}{X} \leq \norm{h - f_{k}}{X} + \norm{f_{k} - f_{0}}{X}
= \norm{h - f_{k}}{X} + \norm{(I - P_{k})f_{0}}{X}.
$$
For the first summand we continue and bound, noticing that~$P_{k}(h -
f_{k}) = h - f_{k}$ as
\begin{align*}
 \norm{h - f_{k}}{X} & = \norm{P_{k}(h-  f_{k})}{} \leq
                       \frac{\norm{\opA(h) - \opA(f_{k})}{Y}}{j(\opA,\X_{k})}\\ 
& \leq  \frac{\norm{\opA(h) - \opA(f_{0})}{Y} + \norm{\opA(f_{0}) -
  \opA(f_{k})}{Y} }{j(\opA,\X_{k})}\\ 
& \leq  \frac{\norm{\opA(h) - \opA(f_{0})}{Y}}{j(\opA,\X_{k})} + \cff \frac{\norm{H^{1/2}(f_{0} -
  f_{k})}{}}{j(\opA,\X_{k})}\\
& =    \frac{\norm{\opA(h) - \opA(f_{0})}{Y}}{j(K,\X_{k})} + \cff \frac{\norm{H^{1/2}(I
  - P_{k})f_{0}}{X}}{j(\opA,\X_{k})}\\
& \leq  \frac{\norm{\opA(h) - \opA(f_{0})}{Y}}{j(\opA,\X_{k})}+ \cff
  \frac{\varrho(H^{1/2},\X_{k})\norm{(I -
  P_{k})f_{0}}{X}}{j(\opA,\X_{k})}.
\end{align*}
Finally, under Assumption~\ref{ass-non-linear} we also have that
$$
\ce j(H^{1/2},\X_{k})\leq j(\opA,\X_{k}) \leq \cff j(H^{1/2},\X_{k}),
$$
which allows to complete the proof.
\end{proof}

As an immediate consequence of the above proposition
we mention the following extension of Proposition \ref{prop:main-bound}.
  \begin{cor}
    Under Assumptions~\ref{ass:relations} and~\ref{ass-non-linear}, we
    have for~$f_0\in\af$ and for $k\in\mathbb N$ large enough, that
$$
\omega_{f_{0}}(\opA^{-1},\X_{k}\cap\mathcal D(\opA),\delta) \leq M\lr{1 + \frac{\cff C_{P}C_{B}}{\ce}}\varphi(s_{k+1}) + \frac{C_{B}}{\ce} \frac{\delta}{\sqrt{s_{k}}}.
$$
\end{cor}

\section{Technical result}
\label{sec:appendix}

The following result is well-known, and taken
from~\cite[Lemma~3]{MR1984890}. A similar result was first used
in~\cite[Lemma~3.4]{MR1257145}. 
\begin{lem}\label{lem:geometry}
  Given~$q,\mu>0$ we consider the function
  $$
  \varphi_{q,\mu}(t) := t^{q}\log^{-\mu}(1/t),\quad    0 < t <1.
  $$
  We assign the related function
  $$
\psi_{q,\mu}(s):= s^{1/q} \log^{\mu/q}(1/s^{1/q}),\quad 0< s < 1. 
$$
Then we have that
$$
\lim_{s\to 0} \frac{\varphi_{q,\mu}^{-1}(s)}{\psi_{q,\mu}(s)} = 1.
$$
\end{lem}

\bibliography{modulus}
\bibliographystyle{imsart-number}

\end{document}